\documentclass[a4paper,12pt]{amsart}
\usepackage[
hmarginratio={1:1},     
vmarginratio={1:1},     
textwidth=15cm,        
textheight=21cm,			
heightrounded,          
]{geometry}

\usepackage[utf8]{inputenc}
\usepackage{graphicx}
\usepackage[percent]{overpic}
\usepackage{pict2e}
\usepackage[pagebackref=false]{hyperref}
\usepackage[dvipsnames]{xcolor}
\definecolor{dark-red}{rgb}{0.4,0.15,0.15}
\definecolor{dark-blue}{rgb}{0.15,0.15,0.4}
\definecolor{dark-green}{rgb}{0.15,0.4,0.15}
\hypersetup{
	colorlinks, linkcolor=dark-red,
	citecolor=dark-blue, urlcolor=dark-green
}

\usepackage{url}
\usepackage{amsfonts, amstext, amsmath, amsthm, amscd, amssymb}
\usepackage[sc]{mathpazo}
\usepackage{caption}
\usepackage{ulem}

\numberwithin{equation}{section}

\theoremstyle{plain}
\newtheorem{theorem}{Theorem}[section]
\newtheorem{proposition}[theorem]{Proposition}
\newtheorem{lemma}[theorem]{Lemma}
\newtheorem{corollary}[theorem]{Corollary}

\theoremstyle{definition}
\newtheorem{definition}[theorem]{Definition}
\newtheorem{example}[theorem]{Example}
\newtheorem{remark}[theorem]{Remark}

\providecommand{\keywords}[1]{\textbf{\textit{Key words and phrases:}} #1}
\providecommand{\subjclass}[1]{\textbf{\textit{2020 Mathematics Subject Classification:}} #1}

\newcommand{\di}{\mathbb{D}}
\newcommand{\s}{\mathbb{S}}
\newcommand{\ZZ}{\mathbb{Z}}
\newcommand{\RR}{\mathbb{R}}

\graphicspath{ {figures/} } 

\title{On 4-dimensional 3-handle attachments}

\author{Eva Horvat}
\address{University of Ljubljana, Faculty of Education, Kardeljeva plo\v s\v cad 16, 1000 Ljubljana, Slovenia}
\email{\href{mailto:eva.horvat@pef.uni-lj.si}{eva.horvat@pef.uni-lj.si}}

\author{Micha{\l} Jab{\l}onowski}

\address{Institute of Mathematics, Faculty of Mathematics, Physics and Informatics, University of Gda\'nsk, 80-308 Gda\'nsk, Poland}

\email{\href{mailto:michal.jablonowski@gmail.com}{michal.jablonowski@gmail.com}}


\date{\today}

\begin{document}
	
\maketitle


\begin{abstract}
Kirby diagrams for smooth four-dimensional manifolds typically depict only the 1- and 2-handles, omitting the 3-handles. In this work, we investigate 3-handle attachments and provide tools to explicitly include them in handle diagrams. We show a set of moves involving 3-handles to extend the classical Kirby calculus. Under assumptions and the condition that the number of 3-handles equals the rank of the spherical part of the specific boundary’s second homology group, we establish a homological criterion that identifies a geometric basis of disjoint embedded spheres in the boundary corresponding to 3-handle attachments, yielding a uniqueness theorem for 3-handle attachments.
\end{abstract}

\section{Introduction}\label{Introduction}
We work in the smooth ($C^{\infty}$) category. All manifolds will be assumed to be compact. We study the attachment of $3$-handles in smooth $4$-manifolds. A handle decomposition of a smooth connected $4$-manifold $X$ is often represented by a Kirby diagram that specifies the attachment of handles of index $1$ and $2$ to the boundary sphere of the $0$-handle. Handles of index $3$ have trivial framing, so their attachments are completely specified by the embedding of their attaching $2$-spheres into the $3$-manifold $\partial X_2$. However, these attaching spheres are rarely drawn for two reasons (other than the fact that it is not so easy to visualize a nontrivial $2$-sphere in the $3$-dimensional picture):

\begin{enumerate}
	\item If $X$ is closed, we may assume that there is a unique $4$-handle, thus by duality, the union of the $3$- and $4$-handles is diffeomorphic to the handlebody $\natural _{k}\,\s^{1}\times \di ^{3}$, where $k$ is the number of $3$-handles. By \cite{LP72}, any self-diffeomorphism of $\# _{k}\,\s^{1}\times \s^{2}$ extends over $\natural _{k}\,\s^{1}\times \di^{3}$, so if $\partial X_2=\# _{k}\,\s^{1}\times \s^{2}$, there is a unique closed $4$-manifold that can be obtained by attaching the $3$- and $4$-handles to $X_2$. 
	\item If $\partial X\neq \emptyset $ is connected and $X$ is simply connected, then it follows from \cite{Tr82} that the $4$-manifold $X$ is uniquely determined by $X_2$ and the number of $3$-handles. 
\end{enumerate}


This paper is organized as follows. 
In Section \ref{Preliminaries}, we review some background on handle decompositions and handle calculus. In Section \ref{4-manifolds}, we recall the basic facts of Kirby calculus and include some $3$-manifold results that will be used to study $4$-manifold boundaries in the further sections.\\
Section \ref{sec_diagrams} introduces a diagrammatic framework to represent $3$-handle attachments in Kirby diagrams, explaining how to draw and manipulate the attaching $2$-spheres of $3$-handles. It defines an extended notation, which is the usual Kirby diagram (specifying the attaching spheres of $0$-, $1$-, and $2$-handles of a $4$-handlebody), augmented with embedded $2$-spheres indicating the attaching spheres of $3$-handles. Subsection \ref{sub_moves} presents a set of moves extending Kirby calculus to diagrams with $3$-handle information, which combine standard handle slides and isotopies of $1$- and $2$-handles with moves involving the $3$-handle attaching spheres. \\
In Section \ref{sec_Uniqueness}, we prove that in a handle decomposition of a $4$-manifold $W$ with boundary $M'\#(\#\;\s^{1}\times \s^{2})$ with $M'$ irreducible, whose rank of $H_2^{\mathrm{sph}}(\partial W_2)$ equals the number of $3$-handles, the $3$-handle attachments correspond to a geometric basis of $H_2^{\mathrm{sph}}(\partial W_2)$. The final Subsection \ref{An_example} applies these ideas to an example.


\section{Preliminaries on handle decompositions}\label{Preliminaries}

One can find an introduction to the topic of handle decompositions (especially in dimension four) in \cite{Akb16, GS99, Kir89} and a more recent paper \cite{Bar23}.

We will denote by $\di^n=\{x\in \mathbb{R}^{n}:|x|\leq 1\} $ the $n$-dimensional disc, and by $\s^n=\partial \di^{n+1}=\{x\in \mathbb{R}^{n+1}:|x|=1\}$ the $n$-dimensional sphere. The symbol $\cong$ will denote equivalence in a given context: diffeomorphism for smooth manifolds, homeomorphism for topological manifolds, or isomorphism for groups.

An \emph{$n$-dimensional $k$-handle} (for $0\leq k \leq n$) is the manifold with corners $h_k = \di^k \times \di^{n-k}$, attached to the boundary of an $n$-manifold $X$ along $\partial \di^k\times \di^{n-k}$. The disc $\di^{k} \times \{0\} \subset h_k$ is called the \emph{core} of the handle, and $ \{0\} \times \di^{n-k} \subset h_k$ is called the \emph{cocore}.

From its corner structure, the boundary of a $k$-handle $\partial h_k = \s^{k-1} \times \di^{n-k} \cup \di^k \times \s^{n-k-1}$ is split into the \emph{attaching region} $\partial_a h_k = \s^{k-1} \times \di^{n-k}$ and the \emph{remaining region} $\partial_r h_k = \di^k \times \s^{n-k-1}$. The sphere $\s^{k-1} \times \{0\} \subset \partial_a h_k$ is called the \emph{attaching sphere}, while $\{0\} \times \s^{n-k-1} \subset \partial_r h_k$ is called the \emph{remaining sphere} or the \emph{belt sphere}. See schematic picture in Figure \ref{handles_general}.

There is a canonical way to smooth the corners, so we may assume the space after attaching handles is a smooth $n$-manifold with boundary. In particular, we can smooth the corners of a handle by attaching it to the boundary of the manifold via a collar neighborhood.

\begin{figure}[h!t]
	\begin{overpic}[scale=0.185]{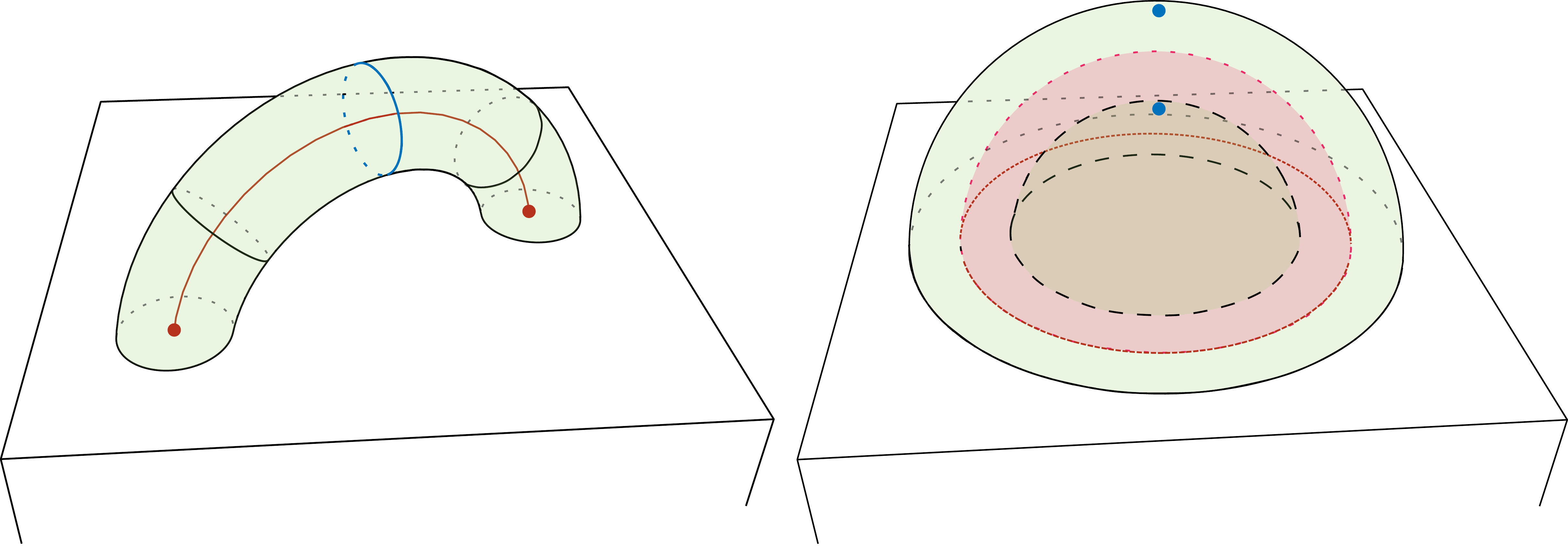}
		\put(42, 33){\textcolor{RoyalBlue}{belt sphere}}
		\put(58,33){\color{RoyalBlue}\vector(2,0.15){15}}
		\put(58,33){\color{RoyalBlue}\vector(3.3,-1.05){15.2}}
		\put(39.5,33){\color{RoyalBlue}\vector(-3.3,-0.8){14.4}}
		\put(45, 22){\textcolor{WildStrawberry}{core}}
		\put(53.8,23){\color{WildStrawberry}\vector(2,0.1){8}}
		\put(42,23){\color{WildStrawberry}\vector(-2,0.1){8.3}}
		\put(37, 29){\textcolor{ForestGreen}{remaining region}}
		\put(59,30){\color{ForestGreen}\vector(2,0.2){3.2}}
		\put(59,30){\color{ForestGreen}\vector(3.3,-2){7.5}}
		\put(39.5,28){\color{ForestGreen}\vector(-3.3,-1.5){3.6}}
		\put(36, 11){\textcolor{BrickRed}{attaching sphere}}
		\put(57,12){\color{BrickRed}\vector(2,1){6.1}}
		\put(34,12){\color{BrickRed}\vector(-4,0.3){22}}
		\put(34,12){\color{BrickRed}\vector(-0.5,11.1){0.39}}
		\put(6, 8){$\partial X$}
		\put(89, 9){$\partial X$}
		\put(6, 2){$X$}
		\put(89, 4){$X$}
	\end{overpic}
	\caption{A schematic picture of adding a $1$-handle (left) and a $2$-handle (right) in the $3$-dimensional case. \label{handles_general}}
\end{figure}

\par 
The number $k$ is called the \emph{index of the handle}. 
Given an embedding $\s^{k-1} \to Y^{n-1}$ with a trivial normal bundle, fix a framing of this bundle. Every framing now corresponds to a map $\s^{k-1} \to O(n-k)$. If we declare the canonical attachment framing (a reference framing) as the $0$-framing, then a choice of any other \emph{framing} corresponds to an element of $\pi_{k-1}(O(n-k)).$
\par 
Every smooth compact manifold $M$ can be constructed by a series of handle attachments. A \emph{handle decomposition} (or \emph{handle structure}) of an $n$-manifold $M$ is a representation of $M$ as a finite sequence of attaching handles. The manifold $M$ can have non-empty boundary $\partial M=\partial_{-}M\coprod \partial_{+}M$, in which case we are talking about a \emph{handle decomposition of $M$ relative to} $\partial_{-}M$ \cite{Mil65}.
\par 
A \emph{$k$-handle attaching map} $\phi$ on an $n$-manifold $M$ is an embedding of the attaching region $\partial_a h_k$ of an $n$-dimensional $k$-handle into $\partial M$. The result of a handle attachment $\phi\colon \partial_a h_k \hookrightarrow \partial M$ is the manifold $M \cup_\phi h_k$, where the handle is glued along the embedded attaching region.
\par
A \emph{$(-1)$-handlebody} is the empty set. A \emph{$k$-handlebody} is obtained from a $(k-1)$-handlebody by successively attaching $k$-handles. A handle decomposition of a manifold $M^n$ is a diffeomorphism to an $n$-handlebody.
By the Isotopy Extension Theorem \cite[p.180]{Hir76}, the manifold $M \cup_\phi h_k$ is determined (up to diffeomorphism) by the isotopy class of the attaching sphere and the framing of $h_k$.

\subsection{Handle calculus}

\begin{remark}
	We can perform an \emph{attachment sequence reordering}, i.e., if the handle $h'$ is attached after the handle $h$, so that the attaching sphere of $h'$ misses the belt sphere of $h$, then we may reorder the handle attachments so that $h$ is attached after $h'$ without changing the diffeomorphism type of the resulting manifold. Thus, we can arrange the handles to be attached in the increasing order of their indices. In the remainder of this paper, we will always assume our handle decompositions satisfy this condition. Given a handle decomposition of a manifold $M$, we will denote by $M_{k}$ the $k$-handlebody that is the result of the attachments of all handles of indices $i$ with $0\leq i\leq k$.
\end{remark}

\begin{remark}
	A $k$-handle can be isotoped over another $r$-handle (for $r\leq k$) in a canonical way; this is called a \emph{handle slide}. Since the attaching order for handles of the same level can be changed arbitrarily, any $k$-handle can be slid over any other $k$-handle.
\end{remark}

\begin{definition}
	A $k$-handle $h_k$ and a $(k+1)$-handle $h_{k+1}$ are \emph{cancellable} if the attaching sphere of $h_{k+1}$ intersects the remaining sphere (the belt sphere) of $h_k$ transversely in one point. The reverse process of canceling two handles is called \emph{handle creation} (or \emph{handle birth}).
\end{definition}

The terminology of cancellable handles is explained by the statement that $M \cup h_k \cup h_{k+1} \cong M$ if $h_k$ and $h_{k+1}$ are cancellable \cite{Mil65}. This diffeomorphism is called \emph{canceling} the handle pair. There is a well-known complete set of moves to translate between any two diffeomorphic handle decompositions.

\begin{theorem}[\cite{Cer70}]\label{thCer}
	Two handle decompositions of the same compact, smooth manifold are related by a finite sequence of handle attachment map isotopies (including handle slides), attachment sequence reorderings, and handle pair cancellations/creations.
\end{theorem}

\subsection{Surgery}

Surgery is a procedure of changing one ($n$-dimensional) manifold into another of the same dimension, by excising a copy of $\s^k \times \di^{n-k}$ for some $k$, and replacing it with $\di^{k+1} \times \s^{n-k-1}$, glued along the common boundary $\s^k \times \s^{n-k-1}$.

\begin{remark}
	When attaching a $k$-handle to the boundary of a manifold $M$ with an attaching map $\phi$, the attaching region is removed and the remaining region is added.
	The boundary of the result of a handle attachment is thus:
	\[ \partial (M \cup_\phi h_k) = \partial M \backslash \phi(\partial_a h_k) \cup \partial_r h_k \]
	Therefore, attaching handles changes the boundary through surgery.
\end{remark}

\section{Decompositions of $4$-manifolds}\label{4-manifolds}

\subsection{Kirby calculus}

A handle decomposition of a compact $4$-manifold consists of finitely many $k$-handles of indices $k\in\{0,1,2,3,4\}$.

Framings are as follows. A $0$-handle has no attaching region, and the framing of a $1$-handle, as specified by $\pi_{0}(O(3))\cong\ZZ_2$, is unique provided the manifold remains orientable. The framing of a $2$-handle is given by an element of $\pi_{1}(O(2))\cong\ZZ$. A $3$-handle has a unique framing since $\pi_{2}(O(1))\cong 0$, and a $4$-handle has a unique framing since $\pi_{3}(O(0))\cong 0$.

\begin{remark}
	If the manifold $M^4$ is closed or $M_2^4$ has boundary $\#_{m}(\s^1\times \s^2)$ for some $m$ (in case $m=0$ we define it to be $\s^3$), then by the theorem of Laudenbach and Poenaru \cite{LP72} there is no need to specify (or draw) the $3$- and $4$-handles $h$ because $M \cup_\phi h\cong M \cup_\psi h$, for all gluings $\phi, \psi$, so the $3$-handles and $4$-handles are attached uniquely.
\end{remark}

The attachment of $1$- and $2$-handles can be described diagrammatically by drawing classical links in the $3$-sphere $\s^3$, which represents the boundary of the $0$-handle. 
\par 
Thus, we can represent any closed, smooth, oriented, connected $4$-manifold by a $3$-dimensional diagram, with links decorated with dots (see Akbulut's convention \cite{Akb16}) for $1$-handles and framed links for $2$-handles. A \textit{framed link} $L$ is a classical link in $\mathbb{S}^3$ with an integer (the framing) associated to each component $\gamma$ of the link. Framed links represent the attaching spheres of the $2$-handles. The zero framing corresponds to the trivialization of the normal bundle, which would extend over a Seifert surface for $\gamma$ pushed into $\mathbb{D}^4$. A circle with framing $k$ means that the image of the attaching region differs from the zero framing by $k$ full twists around the image of the attaching sphere of the $2$-handle (right-handed for $k>0$, left-handed for $k<0$).

\par
The dotted components can be arranged as split unknots. Every arc of the attaching circle of a $2$-handle that goes over the $1$-handle (along its core interval) should link the dotted circle with linking number one. These links in $\s^3$ can be described by $2$-dimensional diagrams in general position (as in the classical knot theory), called the \emph{Kirby diagrams}. The moves (handle slides and handle pair cancellations/creations) on these diagrams are called \emph{Kirby calculus}. They are shown in Figure \ref{h3moves1} with additional $3$-handle information that we will discuss in Section \ref{sec_diagrams}.

\par 
We can easily read off a presentation for $\pi_1(X)$, namely, the generators correspond to meridians of the dotted circles, and relators are words in terms of generators read as we travel along the framed components. Each time we intersect the disc bounded by a dotted circle, we add the letter corresponding to its generator with an exponent equal to $\pm 1$, depending on which side of the disc we approach the intersection.


\subsection{The $3$-manifold context}

We assume all our $3$-manifolds are closed, connected, and orientable.

\begin{theorem}[Lickorish-Wallace \cite{Lic62, Wal60}]
	Every closed, connected, orientable $3$-manifold can be obtained by surgery on a framed link in $\s ^3$.
\end{theorem}

If a handle decomposition of a $4$-manifold $X$ is given by a Kirby diagram, then the $3$-manifold $\partial X_2$ is obtained exactly by first replacing each dotted circle by a $0$-framed component and then doing surgery on the framed link of the $2$-handle attaching circles. The fundamental group and homology information about $\partial X_2$ can be read off from that link (homology is directly related to the linking matrix).

\begin{corollary}[Lickorish-Wallace \cite{Lic62, Wal60}]
	Every smooth closed, connected, orientable $3$-manifold is the boundary of a smooth, orientable, simply connected $4$-manifold.
\end{corollary}

\begin{remark}
	From the above, we conclude that if we want to find a $4$-manifold $X$ whose $3$-handles are not uniquely specified (i.e., $X$ is not determined by $X_2$ and the number of $3$-handles), then this property cannot be deduced from the intrinsic properties of $\partial X$ if the boundary is connected, because we can find a simply connected $X'$ such that $\partial X = \partial X'$ and we know that $3$-handles in $X'$ are unique.
\end{remark}

\begin{lemma} \label{lemma2} Let $M$ be a compact, connected and orientable $3$-manifold, and let $S\subset M$ be a connected orientable surface, properly embedded in $M$ as a closed subset (if $M$ has boundary, then $\partial S\subset \partial M$). Then $M\backslash S$ has either one or two components. If $H_{1}(M;\ZZ _{2})=0$, then $M\backslash S$ has two components.   
\end{lemma}

\begin{proof} We use the long exact sequence of the pair $(M,M\backslash S)$ with $\ZZ _2$ coefficients:
\begin{xalignat*}{1}
& \ldots \to H_{1}(M;\ZZ _{2})\to H_{1}(M,M\backslash S;\ZZ _2)\to \widetilde{H}_{0}(M\backslash S;\ZZ _2)\to \widetilde{H}_{0}(M;\ZZ _{2})=0
\end{xalignat*} and the duality isomorphism $H_{1}(M,M\backslash S;\ZZ _2)\cong H_{c}^{2}(\text{Int}\; S;\ZZ _{2})\cong \ZZ _{2}$. 
\end{proof}

\begin{definition}
	Let $M$ be a $3$-manifold, and $\Sigma \subset M$ an embedded sphere. If $M\backslash \Sigma $ has two connected components, then $\Sigma $ is called a \underline{separating sphere}, otherwise it is called a \underline{nonseparating sphere}. If $\Sigma $ does not bound a $\mathbb{D}^3$ in $M$, then it is called an \underline{essential} sphere; otherwise it is called \underline{inessential}.
\end{definition}

A $3$-manifold is called \emph{irreducible} if it does not contain an essential $\mathbb{S}^2$. A $3$-manifold is \emph{prime} if it cannot be expressed as a connected sum of two $3$-manifolds, neither of which is $\s^{3}$. Every $3$-manifold is a connected sum of prime manifolds; this representation is unique up to order and orientation-preserving homeomorphisms \cite{Kne29, Mil62}.
A prime orientable $3$-manifold is either irreducible or homeomorphic to $\mathbb{S}^1\times \mathbb{S}^2$.

\begin{theorem}\cite{Lau73}
If two smoothly embedded spheres in a $3$-manifold are (freely) homotopic, then they are isotopic.
\end{theorem}

\begin{proposition}\label{prop:homology-obstructs-homotopy}
	Let \(S_1,S_2\subset M\) be smoothly embedded \(2\)-spheres in a
	\(3\)-manifold \(M\). If \(S_1\) and \(S_2\) are freely homotopic, then,
	after choosing orientations,
	$
	[S_1]=\pm[S_2]
	$ in $H_2(M;\mathbb Z).
	$
	Consequently, if
	$
	[S_1]\neq \pm[S_2]
	\;\text{in }H_2(M;\mathbb Z),
	$
	then \(S_1\) and \(S_2\) are not freely homotopic.
\end{proposition}

\begin{proof}
	Suppose that \(S_1\) and \(S_2\) are freely homotopic. Choose
	parameterizations
	$
	f_i\colon \s^2\to S_i\subset M,
	\; i=1,2,
	$
	and let
	$
	H\colon \s^2\times[0,1]\longrightarrow M
	$
	be a homotopy from \(f_1\) to \(f_2\).
	Give \(\s^2\times[0,1]\) the product orientation. The push-forward under
	\(H\) of a fundamental singular \(3\)-chain of
	\(S^2\times[0,1]\) is a singular \(3\)-chain in \(M\) whose boundary is
	$
	f_{2*}[\s^2]-f_{1*}[\s^2].
	$
	Hence
	$
	f_{1*}[\s^2]=f_{2*}[\s^2]
	\;\text{in }H_2(M;\mathbb Z).
	$
	The parameterizations \(f_i\) determine orientations of the embedded
	spheres \(S_i\). Reversing either orientation changes the corresponding
	homology class by a sign. Therefore, when the spheres are regarded as
	unoriented embedded submanifolds, free homotopy implies
$
	[S_1]=\pm[S_2].
	$
	The final assertion follows by contraposition.
\end{proof}

\bigskip

\subsection{Heegaard diagrams} 

\begin{definition} A \emph{Heegaard splitting} of a closed oriented 3-manifold $M$ is a decomposition of $M$ into two connected 1-handlebodies $M=H_{1}\cup _{h}H_{2}$ for some orientation reversing homeomorphism $h\colon \partial H_{1}\to \partial H_{2}$. The common boundary of both handlebodies is called the \emph{Heegaard surface} of the splitting. 
\end{definition} 

Every closed orientable $3$-manifold $M$ admits a Heegaard splitting, that can be obtained from a self-indexing Morse function $f\colon M\to \RR $ by $H_{1}=f^{-1}(\left (-\infty ,\frac{3}{2}\right ])$ and $H_{2}=f^{-1}(\left [\frac{3}{2},\infty \right ))$ and $\Sigma =\partial H_{1}=f^{-1}\left (\frac{3}{2}\right )$. 

\begin{definition}
Let $H$ be a 3-dimensional $1$-handlebody of genus $g$. A \emph{set of attaching curves for $H$} is a collection $\{\gamma _{1},\ldots ,\gamma _g\}$ of pairwise disjoint simple closed curves in $\partial H$ such that every $\gamma _i$ bounds a disk in $H$ and $\partial H -\{\gamma _{1}, \ldots ,\gamma _{g}\}$ is connected. A \emph{Heegaard diagram}, compatible with a Heegaard splitting $M=H_{1}\cup _{h}H_{2}$, is a triple $(\Sigma ,\alpha ,\beta )$, where $\Sigma $ is a closed orientable surface of genus $g$, $\alpha =\{\alpha _{1},\ldots ,\alpha _g\}$ is a set of attaching curves for $H_1$ and $\beta =\{\beta _{1},\ldots ,\beta _{g}\}$ is a set of attaching curves for $H_2$.  
\end{definition}

A Heegaard diagram $(\Sigma ,\alpha ,\beta )$ allows a reconstruction of the manifold $M$ as follows. Take a tubular neighbourhood $\Sigma \times [0,1]$, glue a $g$-tuple of $2$-handles along the curves $\alpha _{1}\times \{0\},\ldots ,\alpha _{g}\times \{0\}$ and another $g$-tuple of $2$-handles along the attaching curves $\beta _{1}\times \{1\},\ldots ,\beta _{g}\times \{1\}$. The boundary of the resulting ma\-ni\-fold consists of two spheres; by adding $3$-handles along these spheres, we obtain a closed manifold, diffeomorphic to $M$. 

\begin{figure}[h!t]
	\centering
	\begin{overpic}[scale=0.7]{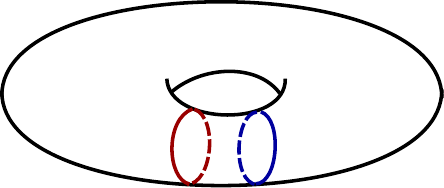}
		\put(30,  8){$\alpha _1$}
		\put(63, 7){$\beta _{1}$}
		\put(80, 20){$\Sigma $}
	\end{overpic}
	\caption{A Heegaard diagram for $\s^{2}\times \s ^{1}$. \label{heeg_diagram}}
\end{figure}

\begin{example} The $3$-manifold $\s^{2}\times \s ^{1}$ may be obtained as the union of two 1-handlebodies of genus 1: $\s^{2}\times \s ^{1}=\left (\di ^{2}\times \s ^{1}\right )\cup _{\textrm{id}}\left (\di ^{2}\times \s ^{1}\right )$. Figure \ref{heeg_diagram} shows a Heegaard diagram, compatible with this Heegaard splitting. 
\end{example}


\section{Diagrams with $3$-handles} \label{sec_diagrams}

We are interested in describing the $3$-handle attachments. Drawings of $3$-handle attaching spheres may be used to compute some $4$-manifold invariants diagrammatically (see \cite{Bar23, Mar09}). In the literature, one can find drawings of the whole system of $3$-handles in Kirby diagrams of some specific examples of $4$-manifolds \cite{HKK86, Bar23, GNS25}. In \cite{HKK86, Bar23}, the $4$-manifolds are closed and $1$-handles are given by the "$3$-ball notation". We want to present and explain a Kirby diagram with $3$-handle attaching spheres in the notation of the latest paper \cite{GNS25}.

Let $X$ be a smooth, connected, orientable compact $4$-manifold with a fixed handle decomposition. Since there is only one way to attach a $4$-handle to a boundary component of $\partial X_3$, the gluing of any $4$-handles does not need to be diagrammatically specified. The Kirby diagram depicts the attaching spheres of the $1$- and $2$-handles in $\s ^{3}$ (the boundary of a $0$-handle).

We will denote by $\mathcal{D}(X_3)$ the Kirby diagram of $X$ (i.e., a diagram of $X_3$ that only specifies the attachment of $0$-, $1$- and $2$-handles). The remaining part of the $3$-sphere after the attachment of $1$- and $2$-handles, i.e., $\s ^{3}\cap \partial X_2$, will be called the ``visible world''. We will denote by $\widehat{\mathcal{D}}(X_3)$ the same diagram, to which embedded $2$-spheres in $\partial X_2$ have been added, indicating the attaching spheres of $3$-handles. 

Of course, only the part of a $3$-handle attaching sphere that intersects the "visible world'' is drawn. This sphere has been isotoped so that we have all the information about how it goes over the $1$- or $2$-handles in the ``invisible part" (as explained in \cite{Bar23}). When a handle decomposition of $X$ is specified, $\mathcal{D}(X)$ will mean $\mathcal{D}(X_2)$ and $\widehat{\mathcal{D}}(X)$ will mean $\widehat{\mathcal{D}}(X_3)$.  


\subsection{Inessential attaching spheres}

\begin{proposition} \label{prop1} Let $X$ be a smooth, compact, connected $4$-manifold, whose handle decomposition is given by a Kirby diagram $\mathcal{D}(X_2)$. \begin{enumerate}
\item[(a)] Adding a $3$-handle to $X_2$ along an inessential sphere in $\partial X_2$ changes the boundary $B=\partial X_2$ to $B\sqcup \s ^{3}$.  
\item[(b)] Any $3$-handle, added to $X_2$ along an inessential sphere in $\partial X_2$, can be cancelled by a $4$-handle.  
\end{enumerate}
\end{proposition}

\begin{figure}[h!t]
	\centering
	\begin{overpic}[scale=0.24]{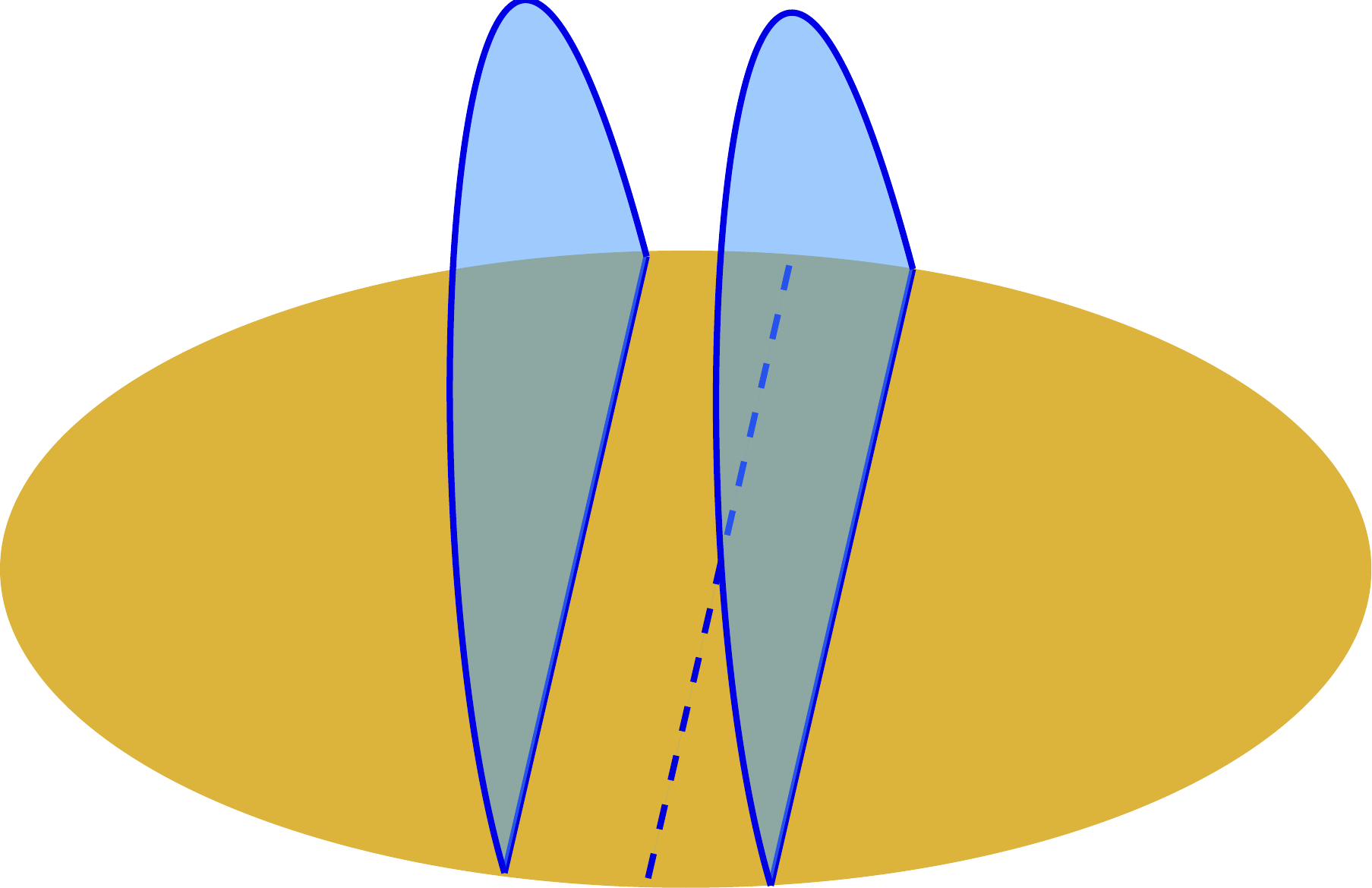}
		\put(44,  8){$\mathcal{S}$}
		\put(18, 20){$\di_{a}^{3}$}
		\put(35, 49){$\di_{b}^{3}$}
		\put(56, 49){$\di_{c}^{3}$}
		\put(68, 20){$\partial X_{2}$}
	\end{overpic}
	\caption{Adding a $3$-handle along an inessential sphere. \label{Figure_nonessential}}
\end{figure}

\begin{proof} \begin{enumerate}
\item[(a)] Let $\mathcal{S}$ be an inessential sphere in $\partial X_2$. Then we may assume $\mathcal{S}$ lies in the boundary of the $0$-handle and is contained in the visible world. Since $\mathcal{S}$ bounds a $3$-ball, the complement of its tubular neighborhood in $\partial X_2$ consists of two connected components $\di _{a}^{3}$ and $\partial X_{2}\backslash (\di _{a}^3)$. When attaching a $3$-handle along $\mathcal{S}$, a tubular neighborhood of $\mathcal{S}$ disappears from $\partial X_2$, while the remaining boundary $\di ^{3}\times \s ^0$ consisting of two balls $\di _{b}^{3}$ and $\di _{c}^{3}$ is added (see Figure \ref{Figure_nonessential}); each ball to one of the components of the remaining $\partial X_2$. The new boundary of the $4$-manifold has two connected components: $\di _{a}^{3}\cup \di _{b}^{3}\cong \s ^{3}$ and $\di _{c}^{3}\cup \partial X_{2}\backslash (\di _{a}^{3})\cong \partial X_{2}$. 
\item[(b)]  We have shown in (a) that by adding a $3$-handle to $X_2$ along an inessential sphere in $\partial X_2$, we obtain a $4$-manifold one of whose boundary components is the $3$-sphere $\di _{a}^{3}\cup \di _{b}^{3}$. The belt sphere of the added $3$-handle consists of two points: the centers of the balls $\di _{b}^{3}$ and $\di _{c}^{3}$. Since the $3$-sphere $\di _{a}^{3}\cup \di _{b}^{3}$ is intersecting the belt sphere exactly at one point, a $4$-handle attached to this $3$-sphere cancels the $3$-handle, and we are left with a manifold diffeomorphic to $X_2$.
\end{enumerate}
\end{proof}

When an attaching sphere of a $3$-handle is inessential, we may draw it in $\widehat{\mathcal{D}}(X_3)$ so that it actually contains the whole ``nontrivial'' part of the Kirby diagram $\mathcal{D}(X_3)$ (the whole link of dotted and framed circles), so the $3$-ball this attaching sphere bounds is ``outside''. Proposition \ref{prop1} implies that there is a move erasing any such sphere from the Kirby diagram. 

\begin{corollary}\label{coro1} Let $X$ be a smooth, compact, connected $4$-manifold with a fixed handle decomposition. If $\partial X_{3}$ is connected, then no $3$-handle in $X_3$ has an inessential attaching sphere. 
\end{corollary}
\begin{proof} This follows directly from Proposition \ref{prop1} (a). 
\end{proof}

To describe the attachment of a $3$-handle in the 3-dimensional context, we may use a Heegaard diagram of $\partial X_2$. Recall that a Heegaard diagram of a 3-manifold $M$ describes a handle decomposition of $M$ based on a self-indexing Morse function $f\colon M\to \RR $. Using the gradient field of $f$, an inessential sphere $\mathcal{S}\subset \partial X_2$ may be pushed onto the boundary of the $0$-handle. Figure \ref{heeg_diagram1} shows such an inessential attaching sphere on two Heegaard diagrams for the 3-sphere $\s ^3$. In general, any Heegaard diagram of a closed 3-manifold $M$ may be stabilized to present $M\,\# \,\s ^{3}\cong M$ and an inessential attaching sphere may be pushed into the $\s ^3$ summand in the same manner.    

\begin{figure}[h!t]
	\centering
	\begin{overpic}[scale=0.7]{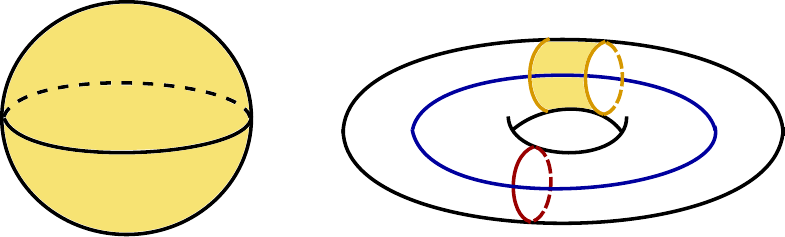}
		\put(60,  9){$\alpha _1$}
		\put(92, 12){$\beta _{1}$}
		\put(96, 21){$\Sigma _{1}$}
                     \put(20, 22){$\Sigma _{0}$}
	\end{overpic}
	\caption{An inessential attaching sphere (yellow) on a Heegaard diagram for $\s ^3$ of genus 0 (left) and of genus 1 (right). Observe that the attaching sphere is contained in one handlebody of the Heegaard splitting. \label{heeg_diagram1}}
\end{figure}


\subsection{Adding $3$-handles to a $1$-handlebody}

The attachment of $k$ $1$-handles to a $0$-handle results in a $1$-handlebody $\natural _{k}\;\s ^{1}\times \di ^{3}$ with boundary $\#_{k}\;\s^{1}\times \s^{2}$. When $k=0$, we let $\#_{0}\;\s^{1}\times \s^{2}$ denote $\s ^3$. Let us discuss possible attachments of $3$-handles to the boundary of this handlebody. 

\begin{proposition} \label{prop3} Consider the $1$-handlebody $X_{1}=\s ^{1}\times \di ^{3}$ with its usual handle decomposition. There are exactly two distinct ways to attach a $3$-handle to $\partial X_1$: \begin{enumerate}
\item along an inessential attaching sphere,
\item along an essential attaching sphere that is nontrivially linked with the attaching sphere of the $1$-handle. 
\end{enumerate} In both cases, we may assume the attaching sphere of the $3$-handle is contained in the visible world of the Kirby diagram (in the "$3$-ball notation"). By adding a $4$-handle to $\partial X_3$, we obtain a manifold $X_{4}\cong X_1$ in case (1) and a closed manifold $X_{4}\cong \s ^{1}\times \s ^{3}$ in case (2). 
\end{proposition}
\begin{proof}
The standard Kirby diagram $\mathcal{D}(\s ^{1}\times \di ^{3})$ describes an attachment of a single $1$-handle $h_1$ to the boundary of the $0$-handle. In the ``$3$-ball notation'', the diagram contains two $3$-balls denoting the attaching region of the $1$-handle. Let $\mathcal{S}$ be an attaching sphere of a $3$-handle in $\partial X_1$. Denote by $f$ a Morse function inducing the handle decomposition of $X_1$. If necessary, we use the flow of the vector field $-\nabla f$ to pull $\mathcal{S}$ off the $1$-handle and into the visible world. Thus, we may assume $\mathcal{S}$ is contained in the Kirby diagram. Let $M=\partial X_{0}\cap \partial X_{1}\subset \s ^{3}$ be the part of the visible world in which $\mathcal{S}$ is embedded. Since $M\cong \di ^{1}\times \s ^{2}$ is simply connected, the sphere $\mathcal{S}$ separates $M$ into two connected components by Lemma \ref{lemma2}. There are two possibilities of the attaching sphere:
\begin{enumerate}
\item If both boundary components of $M$ lie in the same connected component of $M\backslash \mathcal{S}$, then the $3$-handle, attached along $\mathcal{S}$, may be slid off the $1$-handle. Thus $\mathcal{S}$ is inessential and Proposition \ref{prop1} (b) implies it may be cancelled by a $4$-handle.
\item If the boundary components of $M$ lie in distinct components of $M\backslash \mathcal{S}$, then $\mathcal{S}$ represents the nontrivial homology class $[\mathcal{S}]\in H_{2}(\s ^{3}\backslash \partial _{a}h_{1};\ZZ _{2})\cong H_{2}(M;\ZZ _{2})\cong \ZZ _{2}$. 
\end{enumerate}

Let us look more closely at the second case (see Figure \ref{essential_sphere}). In our handle decomposition, the $0$-handle may be represented as $h_{0}=\di ^{1}\times \di ^{3}$, to which the $1$-handle $h_{1}=\di ^{1}\times \di ^{3}$ is added along $\partial \di ^{1}\times \di ^{3}$, creating $X_{1}=\s ^{1}\times \di ^{3}$. The visible part of the boundary $\partial X_{1}$ in the Kirby diagram is $\di ^{1}\times \s ^{2}\subset \partial h_{0}$. After the attachment of our $3$-handle along the attaching sphere $\{0\}\times \s ^{2}$, we obtain the manifold $X_{3}=\di ^{1}\times \s ^{3}\cup _{\partial \di ^{1}\times \di ^{3}} \di ^{1}\times \di ^{3}$ with boundary $\s ^{0}\times \di ^{3}\cup _{\s ^{0}\times \s ^{2}}\di ^{1}\times \s ^{2}\cong \s ^{3}$. Adding a $4$-handle along this $3$-sphere, we obtain the closed $4$-manifold $\s ^{1}\times \s ^{3}$. 
\end{proof}
 
\begin{figure}[h!t]
	\centering
	\begin{overpic}[scale=0.17]{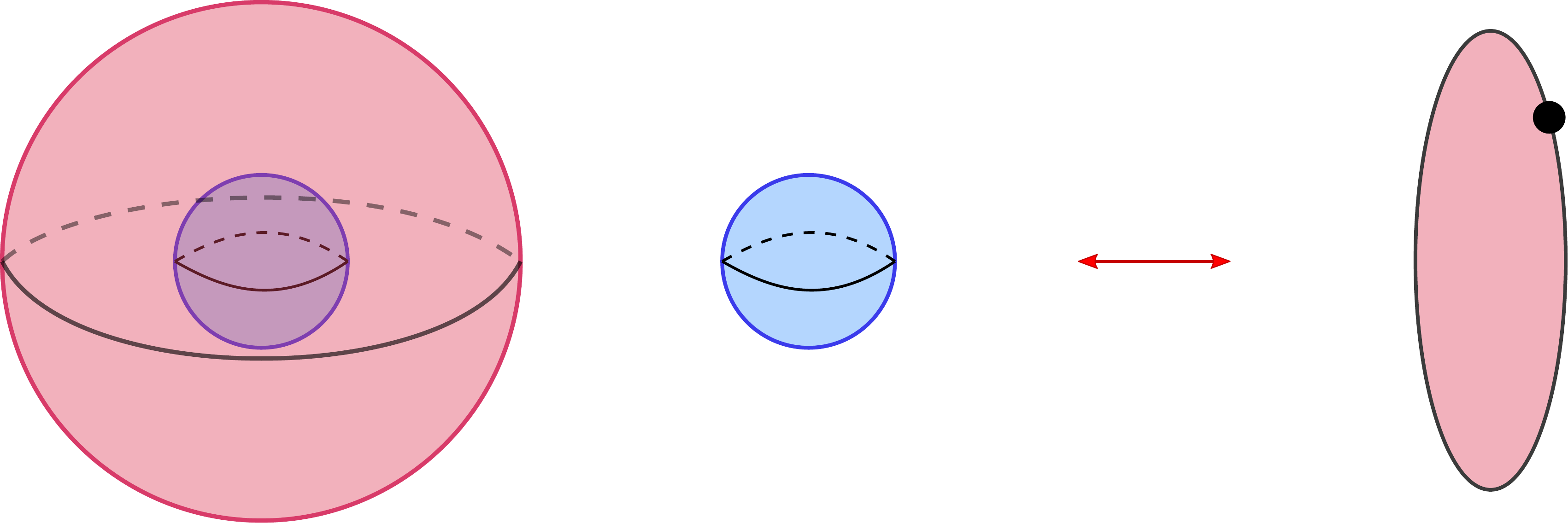}
	\end{overpic}
	\caption{Adding a $3$-handle to $\s ^{1}\times \di ^{3}$ along an essential sphere. \label{essential_sphere}}
\end{figure}

In the ``dotted circle notation'', the $1$-handle of
$\mathcal{D}(\s^1\times\di^3)$ is represented by a dotted unknot
$K\subset\s^3$. Let $\di_a^2\subset\s^3$ be a spanning disk for $K$.
Push the interior of $\di_a^2$ slightly into $\di^4$, while keeping its
boundary fixed, to obtain a properly embedded disk
$
\di_b^2\subset\di^4,\;
\partial\di_b^2=K.
$
The attachment of the $1$-handle may equivalently be described by
carving out an open tubular neighborhood
$
\nu(\di_b^2)\cong \di_b^2\times\di^2
$
from the $0$-handle. Thus
$
\s^1\times\di^3
\cong
\di^4\setminus\operatorname{Int}\nu(\di_b^2),
$
and its boundary decomposes as
\[
\partial(\s^1\times\di^3)
=
\Bigl(
\s^3\setminus\operatorname{Int}\nu(K)
\Bigr)
\cup_{\,K\times\s^1}
\Bigl(
\di_b^2\times\s^1
\Bigr),
\]
where
$
\nu(K)\cong K\times\di^2
$
and the common boundary is the torus
$K\times\s^1$.

Choose a point $\ast\in\s^1$ corresponding to the normal direction
determined by the spanning disk $\di_a^2$, and let
\[
\di_{\mathrm{vis}}^2
=
\operatorname{cl}
\Bigl(
\di_a^2\setminus\nu(K)
\Bigr)
\subset
\s^3\setminus\operatorname{Int}\nu(K).
\]
Then $\di_{\mathrm{vis}}^2$ is a properly embedded disk whose boundary
is the longitude
\[
\partial\di_{\mathrm{vis}}^2
=
K\times\{\ast\}
=
\partial(\di_b^2\times\{\ast\}).
\]
Consequently,
$
\mathcal{S}
=
\di_{\mathrm{vis}}^2
\cup_{\,K\times\{\ast\}}
\bigl(\di_b^2\times\{\ast\}\bigr)
\subset
\partial(\s^1\times\di^3)
$
is an embedded $2$-sphere. It is the standard essential sphere dual to
the $1$-handle. Only the disk
$\di_{\mathrm{vis}}^2$ lies in the visible part of the Kirby diagram.
For simplicity, our convention is to represent this visible disk by
shading the spanning disk $\di_a^2$ bounded by the dotted component in the diagram $\widehat{\mathcal{D}}(\s ^{1}\times \di ^{3})$;
see Figure~\ref{essential_sphere}~(right).

\begin{figure}[h!]
	\centering
	\begin{overpic}[scale=0.75]{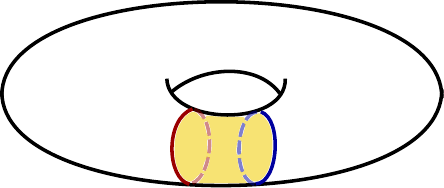}
		\put(30,  8){$\alpha _1$}
		\put(63, 6){$\beta _{1}$}
	\end{overpic}
	\caption{An essential attaching sphere (yellow) on a Heegaard diagram for $\s ^{1}\times \s ^2$. \label{heeg_diagram2}}
\end{figure}

An essential attaching sphere of a $3$-handle, added to $X_{1}=\s ^{1}\times \di ^{3}$, may also be drawn in the Heegaard diagram of $\partial X_1=\s ^{1}\times \s ^{2}$, see Figure \ref{heeg_diagram2}. The sphere $\mathcal{S}$ consists of an annulus on the Heegaard surface, bounded by the attaching circles $\alpha _1$ and $\beta _1$, of the disk that $\alpha _1$ bounds in the handlebody $H_1$ and of the disk that $\beta _1$ bounds in the handlebody $H_2$.


\subsection{Interactions with a $2$-handle}

Given a handle decomposition of a $4$-manifold $X$, an attaching sphere $\mathcal{S}$ of a $3$-handle is embedded in the boundary $3$-manifold $\partial X_2$. Let us discuss possible interactions of the sphere $\mathcal{S}$ with a previously attached $2$-handle $h_2$. There are two regions along which parts of $\mathcal{S}$ may be embedded: \begin{enumerate}
\item a meridional disc $\di _{r}$ in the remaining boundary $\partial _{r}h_2$ that runs parallel to the core of $h_2$.
\item the toroidal boundary $\mathbb{T}_{a}$ of the attaching region $\partial _{a}h_2$,
\end{enumerate}

If the attaching sphere $\mathcal{S}$ intersects the belt sphere of $h_2$ transversely at one point (and thus contains one parallel copy of the disc $\di _{r}$), the $2$- and $3$-handle form a cancelling pair and may thus be removed from the handle decomposition. 

When the $3$-handle is slid over the $2$-handle $h_2$, the attaching sphere $\mathcal{S}$ obtains two parallel copies of the disc $\di _{r}$. This situation is shown in the Figure \ref{attaching_links2} (right). In this case, two copies of the attaching circle of $h_2$ (more precisely, two parallel copies of its longitude on the boundary of the attaching region) bound the visible part of $\mathcal{S}$ in the Kirby diagram.

A visible part of the attaching sphere $\mathcal{S}$ may be nontrivially linked with an attaching circle of a $2$-handle $h_2$. In this case, $\mathcal{S}$ contains an annular region of $\mathbb{T}_a$, see Figure \ref{attaching_links2} (left).

\begin{figure}[h!t]
	\centering
	\begin{overpic}[scale=0.26]{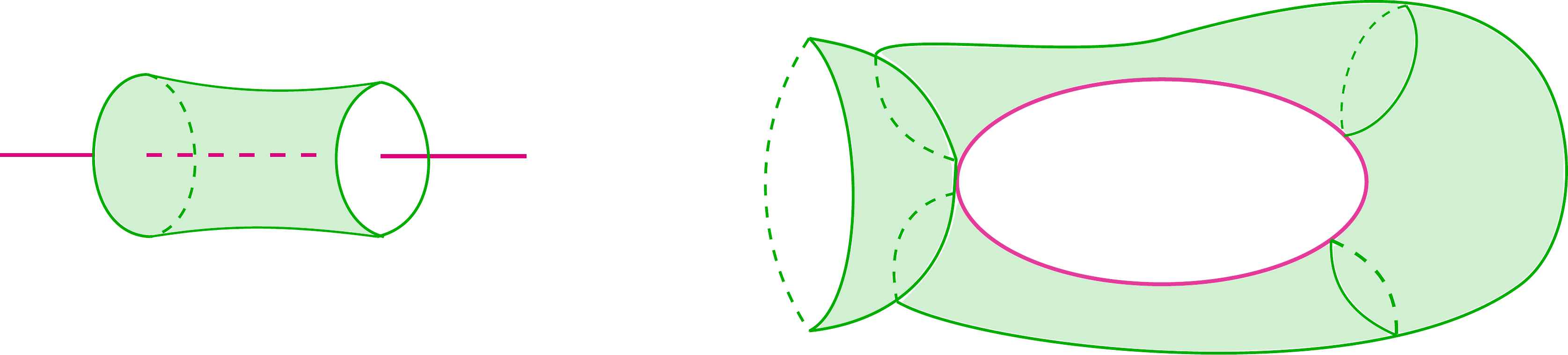}
	\end{overpic}
	\caption{Interactions of a $3$-handle attaching sphere with a previously attached $2$-handle. \label{attaching_links2}}
\end{figure}


\begin{figure}[h!t]
	\begin{center}
		\begin{overpic}[scale=0.66]{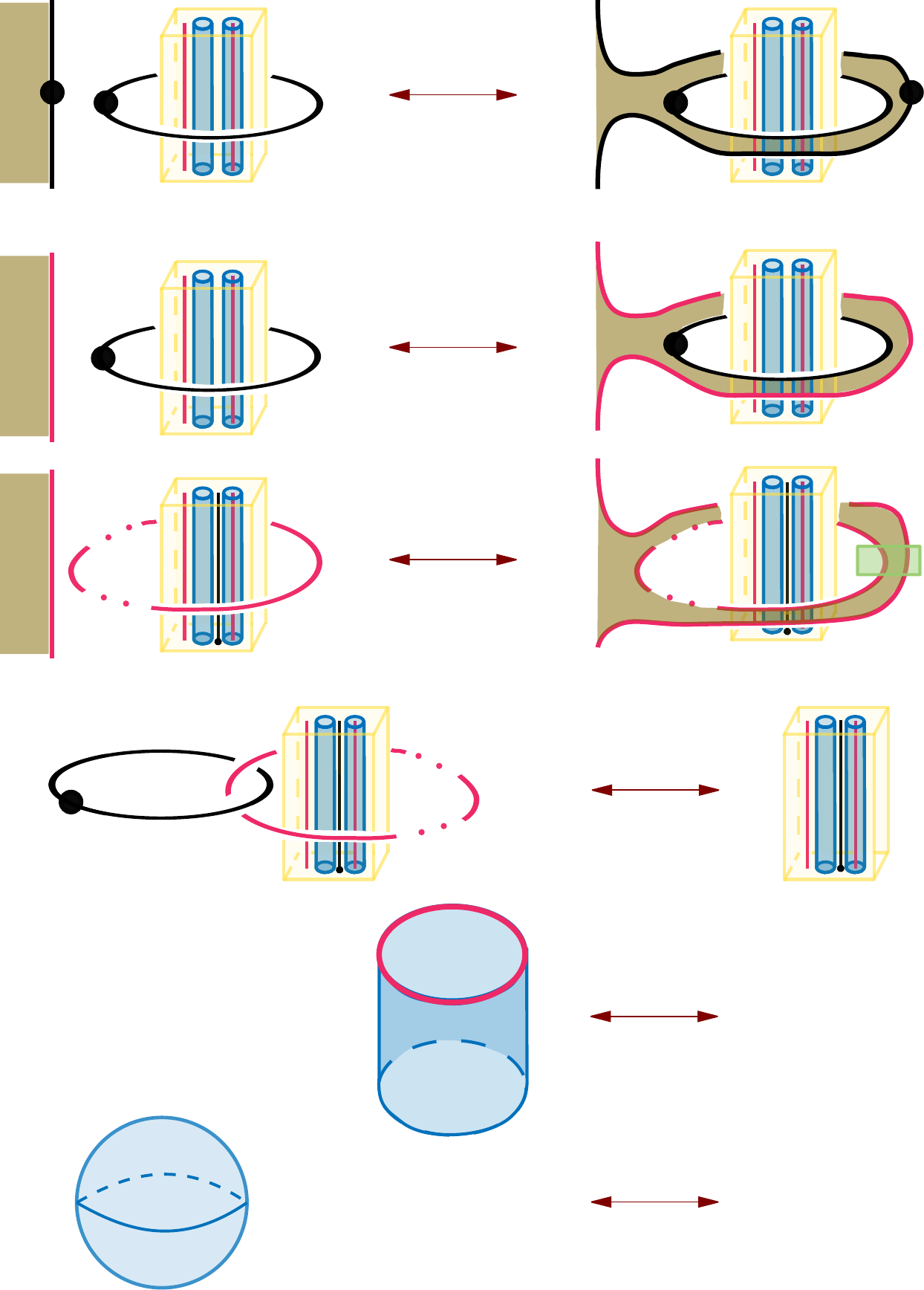}
			\put(30, 95){$(1-1)$ slide}
			\put(30, 75){$(2-1)$ slide}
			\put(30, 59){$(2-2)$ slide}
			\put(45, 41){$(1-2)$ canc.}
			\put(41, 27){$0$}
			\put(45, 23){$(2-3)$ canc.}
			\put(60, 21){$\emptyset$ empty}
			\put(21, 6){$\cup$ cancelling $4$-handle}
			\put(45, 9){$(3-4)$ canc.}
			\put(60, 6){$\emptyset$ empty}
		\end{overpic}
		\caption{Moves on diagrams. Part I.\label{h3moves1}}
	\end{center}
\end{figure}

\begin{figure}[h!t]
	\begin{center}
		\begin{overpic}[scale=0.66]{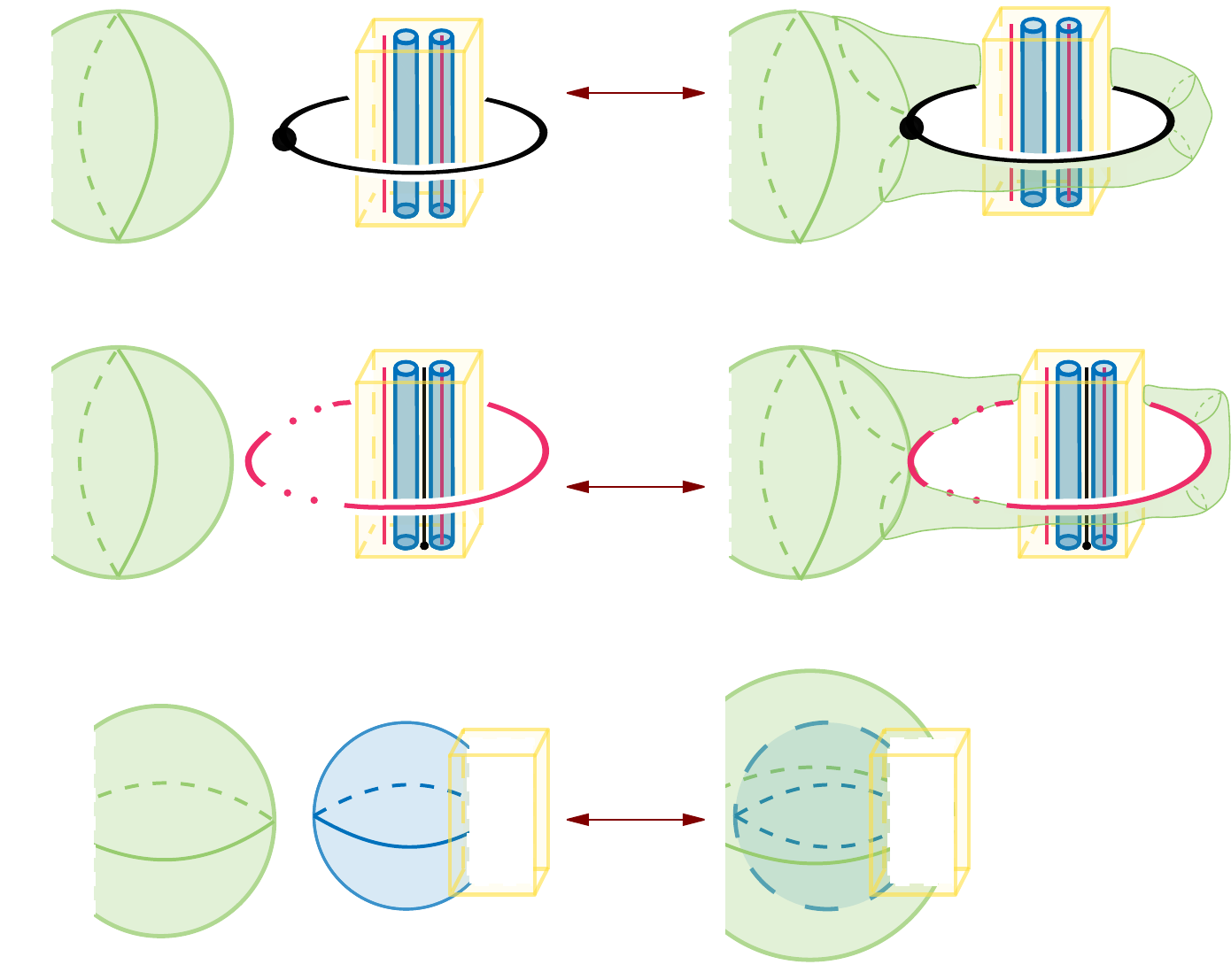}
			\put(43, 75){$(3-1)$ slide}
			\put(43, 48){$(3-2)$ slide}
			\put(43, 23){$(3-3)$ slide}
		\end{overpic}
		\caption{Moves on diagrams. Part II.\label{h3moves2}}
	\end{center}
\end{figure}

\subsection{Moves on diagrams with $3$-handle information} \label{sub_moves}

We present the moves on diagrams with $3$-handle information added in Figures \ref{h3moves1}--\ref{h3moves2}. These moves include the standard Kirby calculus (i.e., drawing only $1$- and $2$-handle attaching information) together with attaching spheres for the $3$-handle attachment, and their interactions that agree with Theorem \ref{thCer}.

Our extended set of moves differs in two substantive ways from Barenz's presentation \cite{Bar23}. First, we work in the dotted circle (Akbulut) notation for $1$-handles rather than the "$3$-ball notation". Second, several of our moves are stated in greater generality: for instance, in the $(2-3)$-handle cancellation, we do not require the visible part of the $3$-handle's attaching sphere to be a disc.

The labels of the moves are above the arrows and refer to the move from left to right. The arrows are drawn in both directions, indicating that the move from right to left is also allowed. The attaching circles of $2$-handles (drawn in magenta) may have an arbitrary integer framing. If only a partial arc of the circle is drawn, the rest of the circle may be arbitrarily knotted. 

Inside the yellow box, there may be an arbitrary combination of $2$- and $3$-handle attaching spheres, except for the situation where the dotted circle is nontrivially linked with another dotted circle, which is not allowed (\cite{GS99}). Moreover, the yellow box in the last picture of Figure \ref{h3moves2} indicates that the attaching spheres of $3$-handles may be bounded by framed or dotted circles that the box is passing through. The brown/olive attaching sphere in Figure \ref{h3moves1} may or may not be present.

In the case of a $(2-2)$-handle slide of a handle $h_1$ (left) over the handle $h_2$ (right), the framing of $h_2$ does not change and the framing of $h_1$ changes to 
$\text{frame}(h_1)\mapsto \text{frame}(h_1)+\text{frame}(h_2) \pm 2\text{lk}(h_1,h_2).$
The green box in Figure \ref{h3moves1} symbolizes the change in the linking between the two curves: the new linking number locally changes by $\text{frame}(h_2)$. 
In the case of a $(2-1)$-handle slide of a $2$-handle $h_1$ (left) over a $1$-handle $h_2$ (right), the framing of $h_1$ changes to $\text{frame}(h_1)\mapsto \text{frame}(h_1) \pm 2\text{lk}(h_1,h_2).$ See \cite[sec. 5.1]{GS99} for the details.


\section{Uniqueness of $3$-handle attachments} \label{sec_Uniqueness}

We know that any closed connected smooth $4$-manifold admits a handle decomposition with only one $4$-handle. In such a handle decomposition, every $3$-handle must attach along an essential sphere by Corollary \ref{coro1}. We have the following.

\begin{proposition}\cite[\S 4]{HKK86}\label{prop2}
Let $W$ be a closed $4$-manifold, whose handle decomposition has one $4$-handle and $\partial W_2=\#_k\;\s^{1}\times \s^{2}$. Then $W$ has exactly $k$ $3$-handles. A $k$-tuple of pairwise disjoint embedded $2$-spheres $\Sigma_1, \ldots, \Sigma_k \subset\partial W_2$ may represent the attaching spheres of these $3$-handles if and only if $[\Sigma_1]$, \ldots, $[\Sigma_k]$ form a basis of $H_2(\partial W_2)\cong\mathbb{Z}^k$. 
\end{proposition}

\begin{example}[An extended Kirby diagram for $\mathbb{T}^2\times \s^{2}$]
We present a Kirby diagram with $3$-handles of the closed manifold $\mathbb{T}^2\times \s^{2}$. Such a diagram has been given in \cite{Bar23}, where the "$3$-ball notation" of $1$-handles was used. We redraw the diagram with the dotted circle notation of $1$-handles. A Kirby diagram of $\mathcal{D}(\mathbb{T}^2\times \s^2)$ is given in Figure \ref{double_T2D2}. Extending this diagram by the attaching spheres of two $3$-handles, we obtain $\widehat{\mathcal{D}}(\mathbb{T}^2\times \s^2)$, presented in Figure \ref{3handles_exe1}.

Observe that the attaching spheres $\partial_{a}h_{3A}$ and $\partial_{a}h_{3B}$ are embedded in $\partial(\mathbb{T}^2\times \s^2)_2$ and form a basis of $H_2(\partial(\mathbb{T}^2\times \s^2)_2)$, because each of them is nontrivially linked with the attaching sphere of one $1$-handle (indicated by a dotted circle); $ \partial(\mathbb{T}^2\times \s^2)_2\cong\#_2\;\s^{1}\times \s^{2}$. Therefore, the two spheres represent the attaching spheres of two $3$-handles in $(\mathbb{T}^2\times \s^2)_2$ by Proposition \ref{prop2}. 
\end{example}

\begin{figure}[h!t]
	\begin{center}
		\begin{overpic}[scale=0.3]{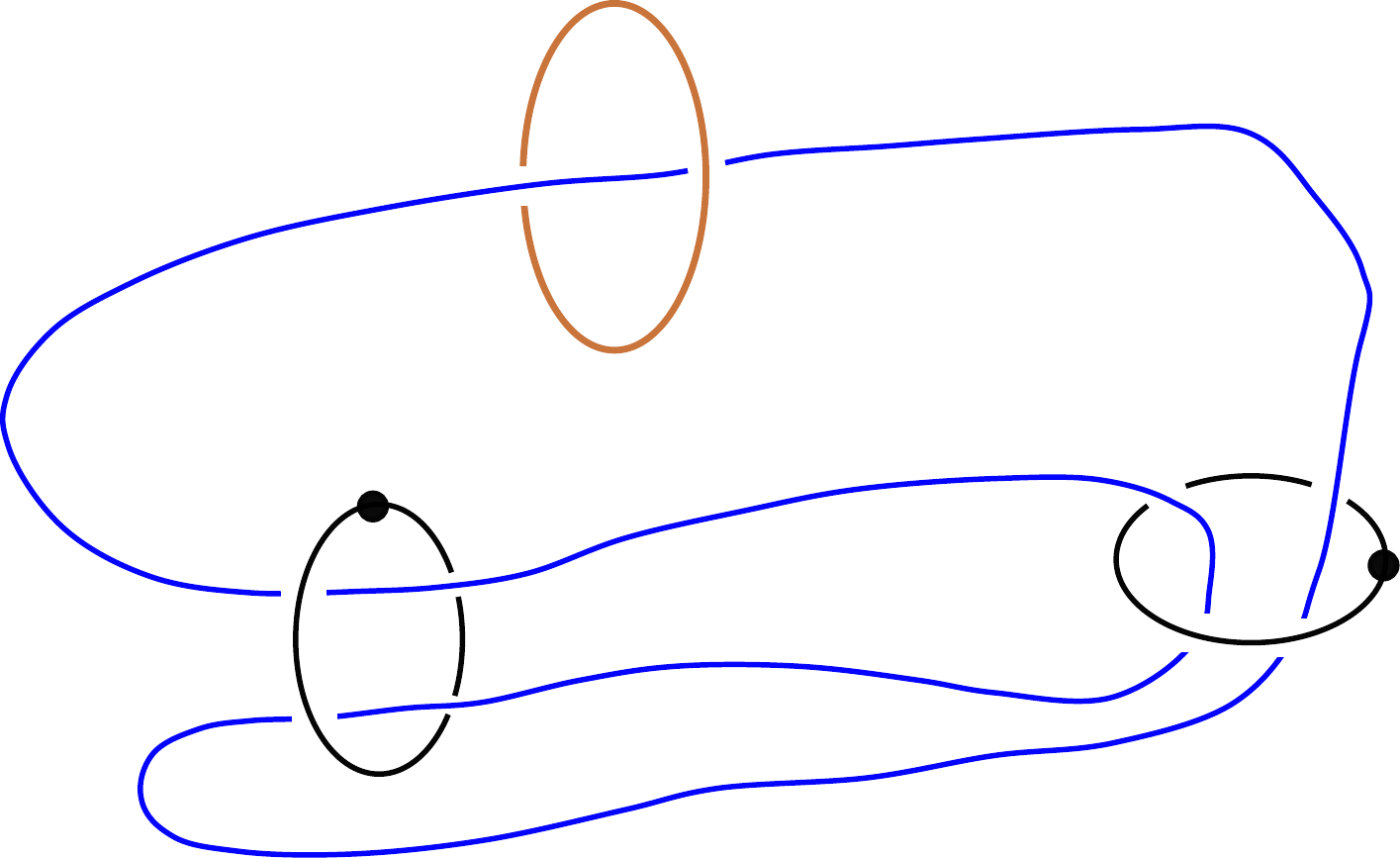}
			\put(60, 16){$0$}
			\put(52, 40){$0$}
		\end{overpic}
		\caption{A Kirby diagram $\mathcal{D}(\mathbb{T}^2\times \s^2)$.\label{double_T2D2}}
	\end{center}
\end{figure}

\begin{figure}[h!t]
	\centering
	\begin{overpic}[scale=0.58]{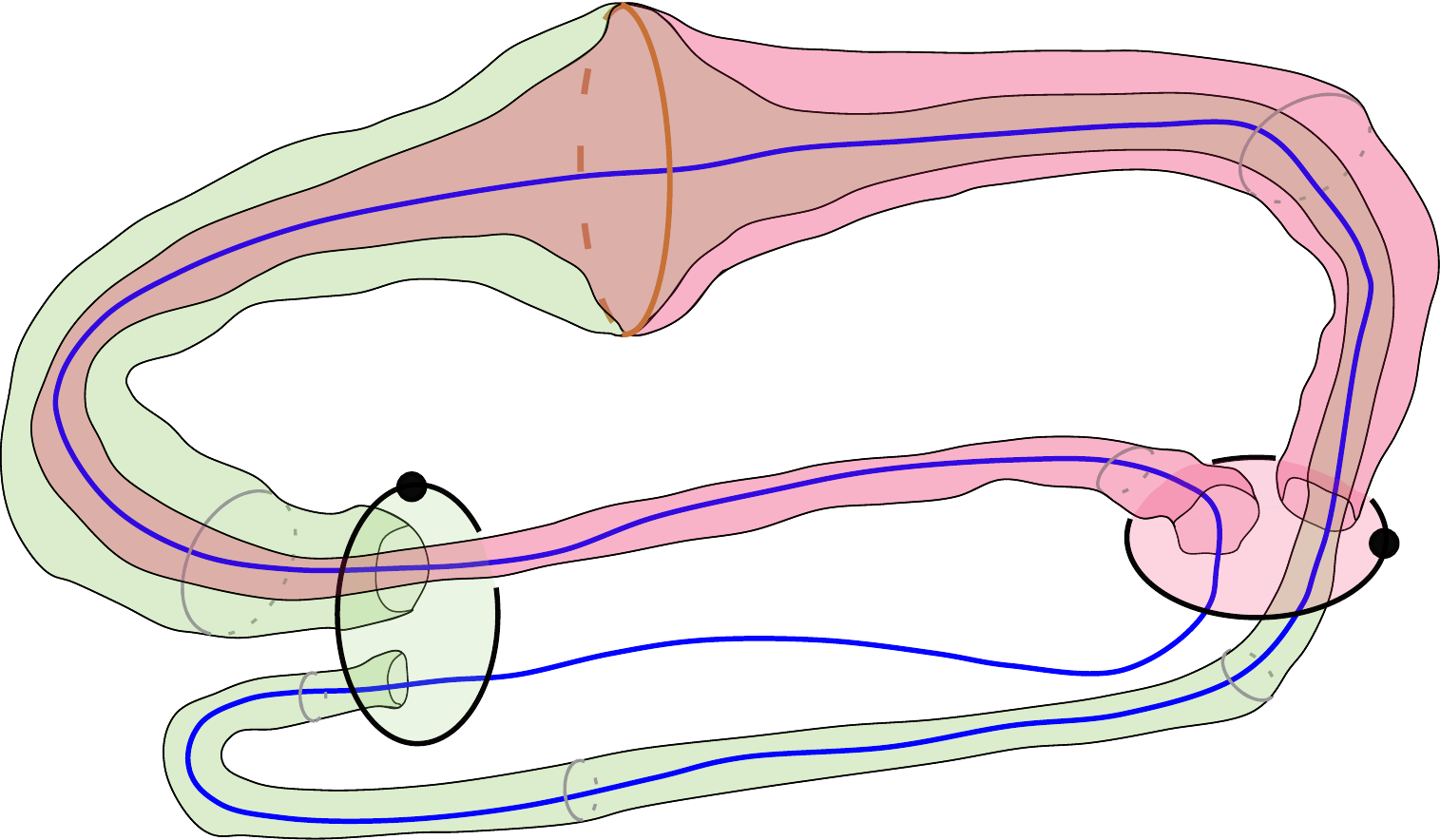}
		
		\put(60, 15){$\partial_{a}h_{2a}$}
		\put(42, 32){$\partial_{a}h_{2b}$}
		
		\put(3, 5){$\partial_{a}h_{3A}$}
		\put(62, 30){$\partial_{a}h_{3B}$}
		
	\end{overpic}
	\caption{A diagram $\widehat{\mathcal{D}}(\mathbb{T}^2\times \s^2)$. \label{3handles_exe1}}
\end{figure}


\subsection{Manifolds with a general boundary} \label{general_boundary}

For a $3$-dimensional manifold $M$, let $H_2^{\mathrm{sph}}(M)$ denote the subgroup of $H_2(M;\mathbb{Z})$, generated by embedded $2$-spheres in $M$: $$H_2^{\mathrm{sph}}(M)=\left<[\Sigma]\in H_2(M): \Sigma\text{ represented by an embedded $\s^2$ }\right>\;.$$  It is isomorphic to the image of the Hurewicz homomorphism $\pi_2(M)\to H_2(M;\mathbb{Z})$ \cite{Hat07}.  We can generalize Proposition \ref{prop2} to $4$-manifolds with a more general boundary, as follows.


\begin{theorem}\label{thm1}
	Let $W$ be a smooth compact connected $4$-manifold with non-empty
	connected boundary, equipped with a fixed handle decomposition without
	$4$-handles. Assume that the number of $3$-handles is $k$ and that
	$$
	\partial W_2 \cong M' \# \Bigl(\#_{k}(\s^1\times \s^2)\Bigr),
	$$
	where $M'$ is irreducible. Let $\Sigma_1,\dots,\Sigma_k \subset \partial W_2$ be pairwise disjoint
	smoothly embedded $2$-spheres. Then the following are equivalent:
	
	\begin{enumerate}
		\item Up to permutation and $3$--$3$ handle slides, the spheres
		$\Sigma_1,\dots,\Sigma_k$ may be isotoped to the actual attaching
		spheres of the $k$ $3$-handles of $W$.
		
		\item The complement
		$
		\partial W_2 \setminus \operatorname{Int}\nu(\Sigma_1\cup\cdots\cup\Sigma_k)
		$
		is connected.
		
		\item The classes
		$
		[\Sigma_1],\dots,[\Sigma_k]
		$
		form a $\mathbb Z$-basis of $H_2^{\mathrm{sph}}(\partial W_2)$.
	\end{enumerate}
	
	Moreover, whenever these equivalent conditions hold, simultaneous
surgery on $\Sigma_1$, $\dots,$ $\Sigma_k$ produces a closed $3$-manifold
diffeomorphic to $M'$.
\end{theorem}


\begin{definition}
	Let
	$
	\mathcal{S}=\{S_1,\ldots,S_k\}
	$
	be a collection of pairwise disjoint embedded \(2\)-spheres in a
	\(3\)-manifold \(M\). A \emph{sphere slide of \(S_i\) over \(S_j\)}
	(\(i\neq j\)) is defined as follows. Choose an embedded arc
	$
	\gamma\subset M
	$
	whose endpoints lie on \(S_i\) and \(S_j\), respectively, and whose
	interior is disjoint from \(\mathcal S\). Replace \(S_i\) by the
	sphere obtained by tubing \(S_i\) to a parallel copy of \(S_j\) along
	\(\gamma\). All other components of \(\mathcal{S}\) are left
	unchanged.
\end{definition}

\begin{lemma}[Spotted-ball slide lemma]
	\label{lem:spotted-ball}
	Let \(R\) be a \(3\)-ball and let
	$
	B_0,B_1,\ldots,B_r\subset\operatorname{Int}R
	$
	be pairwise disjoint \(3\)-balls. Put
	\[
	P
	=
	R\setminus
	\bigcup_{\ell=0}^{r}\operatorname{Int}B_\ell.
	\]
	Then the outer boundary sphere \(\partial R\) can be obtained from
	\(\partial B_0\) by successively tubing \(\partial B_0\) to
	$
	\partial B_1,\ldots,\partial B_r
	$
	along pairwise disjoint arcs in \(P\).
	
	In particular, when the spheres \(\partial B_\ell\) are boundary
	copies of the components of a sphere system obtained by cutting a
	closed \(3\)-manifold along that system, these tubings are induced by
	sphere slides.
\end{lemma}

\begin{proof}
	Choose an embedded tree \(T\subset R\) whose vertices consist of one
	point in each ball \(B_\ell\), and whose edges are disjoint from the
	balls except at their endpoints. We may choose \(T\) so that its edges
	are pairwise disjoint.
	
	A closed regular neighbourhood
	$
	N
	=
	\nu\left(
	B_0\cup\cdots\cup B_r\cup T
	\right)
	$
	is a \(3\)-ball. Starting with \(B_0\), adjoining an edge of \(T\)
	together with the ball at its other endpoint changes the boundary by
	tubing the previous boundary sphere to the boundary of that ball.
	Consequently, \(\partial N\) is obtained from \(\partial B_0\) by the
	stated sequence of tubings.
	
	By the smooth Schoenflies theorem, the embedded ball
	\(N\subset\operatorname{Int}R\) is unknotted in \(R\). Hence
	$
	\overline{R\setminus\operatorname{Int}N}
	\cong S^2\times[0,1],
	$
	so \(\partial N\) is isotopic to \(\partial R\) in \(P\).
\end{proof}


\begin{definition}[Boundary slide]
	\label{def:boundary-slide}
	Let \(Q\) be a compact \(3\)-manifold whose boundary components are
	\(2\)-spheres, and let
	$
	\mathcal S=\{S_1,\ldots,S_k\}
	$
	be a collection of pairwise disjoint embedded \(2\)-spheres in
	\(\operatorname{Int}Q\).
	
	Let \(F\) be a component of \(\partial Q\), and let
	$
	F'\subset\operatorname{Int}Q
	$
	be a parallel copy of \(F\). Choose an embedded arc
	$
	\gamma\subset Q
	$
	with one endpoint on \(S_i\), the other endpoint on \(F'\), and
	\[
	\operatorname{Int}\gamma
	\cap
	\bigl(\mathcal S\cup F'\bigr)
	=
	\varnothing.
	\]
	A \emph{boundary slide of \(S_i\) over \(F\)} replaces \(S_i\) by
	the sphere obtained by tubing \(S_i\) to \(F'\) along \(\gamma\).
	All other components of \(\mathcal S\) are left unchanged.
	The slide is performed in the interior of \(Q\), so the boundary
	\(\partial Q\) remains fixed pointwise.
\end{definition}

\begin{lemma}[Relative uniqueness of complete sphere systems]
	\label{lem:relative-complete-sphere-systems}
	Let \(Q\) be a compact, connected, orientable \(3\)-manifold whose
	boundary is a possibly empty union of \(2\)-spheres. Let
	\(\widehat Q\) be the closed \(3\)-manifold obtained by capping every
	component of \(\partial Q\) with a \(3\)-ball.
	
	Assume that
	$
	\widehat Q
	\cong
	M'\#\bigl(\#_{k}(\s^1\times \s^2)\bigr),
	$
	where \(M'\) is closed, connected, orientable, and irreducible.
	
	Let
	$
	\mathcal S=\{\Sigma_1,\ldots,\Sigma_k\},
	$
	$
	\mathcal A=\{A_1,\ldots,A_k\}
	$
	be two collections of pairwise disjoint smoothly embedded
	\(2\)-spheres in \(\operatorname{Int}Q\). Assume that:
	
	\begin{enumerate}
		\item every component of \(\mathcal S\) and \(\mathcal A\) is
		essential in \(\widehat Q\);
		\item both complements
		$
		Q\setminus\operatorname{Int}\nu(\mathcal S)
		$ and $
		Q\setminus\operatorname{Int}\nu(\mathcal A)
		$
		are connected.
	\end{enumerate}
	
	Then \(\mathcal S\) and \(\mathcal A\) are related by a finite
	sequence of:
	
	\begin{enumerate}
		\item ambient isotopies of \(Q\) fixed on \(\partial Q\);
		\item permutations of the components;
		\item sphere slides of one component over another;
		\item boundary slides over components of \(\partial Q\).
	\end{enumerate}
\end{lemma}

\begin{proof}
	We argue by induction on \(k\). The assertion is immediate for
	\(k=0\). Assume \(k>0\).
	Write
	$
	\partial Q=F_1\sqcup\cdots\sqcup F_m.
    $
	Let \(D_a\) be the \(3\)-ball used to cap \(F_a\) in the construction
	of \(\widehat Q\).
	Put
	$
	C_{\mathcal S}
	=
	Q\setminus\operatorname{Int}\nu(\mathcal S).
	$
	Its boundary consists of the original boundary components \(F_a\)
	and two copies
	$
	\Sigma_i^+,\Sigma_i^-
	$
	of each \(\Sigma_i\). Cap \(\Sigma_i^\pm\) by \(3\)-balls
	\(B_i^\pm\), and cap \(F_a\) by \(D_a\). Denote the resulting closed
	manifold by \(\widehat C_{\mathcal S}\).
	
	Reversing the cutting-and-capping operation along the components of
	\(\mathcal S\) adds \(k\) $3$-dimensional \(1\)-handles. Hence
	$
	\widehat Q
	\cong
	\widehat C_{\mathcal S}
	\#
	\bigl(\#_k(\s^1\times \s^2)\bigr).
	$
	By the Kneser--Milnor uniqueness theorem,
	$
	\widehat C_{\mathcal S}\cong M'.
	$
	In particular, \(\widehat C_{\mathcal S}\) is irreducible.
	
	\medskip
	\noindent
	\emph{Step 1: making the two sphere systems disjoint.}
	
	Put \(\mathcal S\) and \(\mathcal A\) in transverse position and set
	\[
	I(\mathcal S,\mathcal A)
	=
	\sum_{i,j}
	\#\pi_0(\Sigma_i\cap A_j).
	\]
	Suppose that \(I(\mathcal S,\mathcal A)>0\). Choose a circle of
	intersection that is innermost on some \(A_j\). It bounds a disk
	$
	D\subset A_j
	$
	such that
	$
	\operatorname{Int}D\cap\mathcal S=\varnothing.
	$
	Suppose that
	$
	\partial D\subset\Sigma_i.
	$
	
	After cutting \(Q\) along \(\mathcal S\), the disk \(D\) becomes a
	properly embedded disk in \(C_{\mathcal S}\), whose boundary lies,
	say, on \(\Sigma_i^+\). In \(\widehat C_{\mathcal S}\), take a small
	regular neighbourhood
	$
	N=\nu(B_i^+\cup D).
	$
	The manifold \(N\) is diffeomorphic to \(\s^2\times[0,1]\), and its
	two boundary components are parallel copies of the two spheres
	obtained by surgery on \(\Sigma_i\) along \(D\).

	Since $\widehat{C}_{S}$ is irreducible, at least one component of
	$
	\widehat{C}_{S}\setminus \operatorname{Int}N
	$
	adjacent to a component of $\partial N$ is a $3$-ball. Let $R$ be such
	a ball, and let $T\subset \partial N$ be the component satisfying
	$T=\partial R$. We choose the regular neighbourhood $N$ sufficiently
	small so that it is disjoint from every cap ball except $B_i^{+}$.
	
	We distinguish two cases.
	
	Suppose first that $B_i^{-}\subset R$. Apply Lemma~\ref{lem:spotted-ball} to the ball
	$R$, taking $B_i^{-}$ as the distinguished inner ball. Since
	$B_i^{+}\subset \operatorname{Int}N$ and
	$R\cap\operatorname{Int}N=\varnothing$, the ball $B_i^{+}$ is not
	contained in $R$. It follows that the surgery sphere
	$
	T=\partial R
	$
	is obtained from $\partial B_i^{-}$ by successively tubing it to the
	boundaries of the other cap balls contained in $R$. In particular,
	this construction does not require tubing $\partial B_i^{-}$ to the
	opposite copy $\partial B_i^{+}$ of $\Sigma_i$.
	
	Suppose now that $B_i^{-}\not\subset R$. The manifold
	$
	R\cup_{T}N
	$
	is a $3$-ball whose boundary is the other component
	$T'\subset\partial N$. Moreover, by the choice of $N$, the ball
	$B_i^{-}$ is not contained in $R\cup_{T}N$. We may therefore apply
	Lemma~\ref{lem:spotted-ball} to $R\cup_{T}N$, taking $B_i^{+}$ as the distinguished
	inner ball. The sphere
	$
	T'=\partial(R\cup_{T}N)
	$
	is then obtained from $\partial B_i^{+}$ by successively tubing it to
	the boundaries of the other cap balls contained in $R\cup_{T}N$.
	Again, no tubing to the opposite copy $\partial B_i^{-}$ of
	$\Sigma_i$ is required.
	
	Consequently, in either case, one of the two spheres obtained by
	surgery on $\Sigma_i$ along $D$ is obtained from a boundary copy of
	$\Sigma_i$ by tubings only to cap balls of the following types:
	$
	B_{\ell}^{\pm},$ $\ell\neq i,
	$
	or
	$
	D_a,
	$
	where $D_a$ is the ball used to cap a boundary component $F_a$ of
	$Q$. After regluing the two boundary copies of every component of
	$S$, tubing to $\partial B_{\ell}^{\pm}$ is induced by a sphere slide
	of $\Sigma_i$ over $\Sigma_{\ell}$, while tubing to $\partial D_a$ is
	induced by a boundary slide of $\Sigma_i$ over $F_a$. Thus one of the
	two surgery spheres may replace $\Sigma_i$ by a finite sequence of
	sphere slides, boundary slides, and an ambient isotopy fixed on
	$\partial Q$.
	
	Choose this surgery sphere and push the copy of the surgery disk $D$
	slightly off $A_j$. Since $\partial D$ was chosen innermost on $A_j$,
	we have
	$
	\operatorname{Int}D\cap \mathcal S=\varnothing.
	$
	Moreover, since the components of $A$ are pairwise disjoint, the
	interior of $D$ is disjoint from every component of
	$\mathcal A\setminus\{A_j\}$. Hence the replacement introduces no new circles
	of intersection with $\mathcal A$, whereas the circle $\partial D$ disappears.
	Therefore,
	$
	I(\mathcal S,\mathcal A)
	$
	strictly decreases.
	
	Sphere slides and boundary slides preserve the connectedness of the
	complement of the sphere system. Furthermore, after cutting along the
	sphere system and capping all resulting spherical boundary
	components, the capped complement remains diffeomorphic to
	$\widehat{C}_{S}$. Thus the same argument may be applied repeatedly.
	After finitely many repetitions, we obtain
	$
	\mathcal S\cap \mathcal A=\varnothing.
	$

	\medskip
	\noindent
	\emph{Step 2: matching one sphere.}
	
	The spheres \(A_j\) may now be regarded as embedded spheres in
	$
	\widehat C_{\mathcal S}\cong M'.
	$
	Since \(M'\) is irreducible, every \(A_j\) bounds a \(3\)-ball in
	\(\widehat C_{\mathcal S}\).
	
	Choose a component, relabelled \(A_1\), and a ball \(R\) bounded by
	\(A_1\), such that
	\[
	\operatorname{Int}R
	\cap
	(\mathcal A\setminus\{A_1\})
	=
	\varnothing.
	\]
	
	The ball \(R\) contains at least one of the cap balls \(B_i^\pm\).
	Indeed, if it contained none of them, then restoring the
	three-dimensional \(1\)-handles corresponding to \(\mathcal S\)
	would not change \(R\), so \(A_1\) would bound a \(3\)-ball in
	\(\widehat Q\). This contradicts the essentiality of \(A_1\).
	
	Moreover, some pair \(B_i^+,B_i^-\) is separated by \(A_1\).
	Otherwise, for every \(i\), the balls \(B_i^+\) and \(B_i^-\) would
	lie on the same side of \(A_1\). Restoring all the corresponding
	\(1\)-handles would then leave \(A_1\) separating in \(\widehat Q\).
	
	On the other hand,
	$
	Q\setminus\operatorname{Int}\nu(\mathcal A)
	$
	is connected, so \(A_1\) is nonseparating in \(Q\), and therefore
	also in \(\widehat Q\). This is a contradiction.
	
	After interchanging the signs if necessary, choose \(i\) such that
	\[
	B_i^+\subset R,
	\quad
	B_i^-\not\subset R.
	\]
	Apply Lemma~\ref{lem:spotted-ball} to \(R\), using \(B_i^+\) as the
	distinguished inner ball. The lemma expresses
	$
	A_1=\partial R
	$
	as the result of tubing \(\partial B_i^+\) to the boundaries of all
	the other cap balls contained in \(R\).
	
	Tubing to a cap \(B_\ell^\pm\) gives a sphere slide of
	\(\Sigma_i\) over \(\Sigma_\ell\), while tubing to \(D_a\) gives a
	boundary slide over \(F_a\). Since \(B_i^-\not\subset R\), no tubing
	to the opposite copy of \(\Sigma_i\) is required.
	
	It follows that, after sphere slides, boundary slides, and an ambient
	isotopy fixed on \(\partial Q\), the sphere \(\Sigma_i\) agrees with
	\(A_1\). Since the moves are supported in \(R\), they may be chosen
	disjoint from
	$
	\mathcal A\setminus\{A_1\}.
	$
	After permuting the components of \(\mathcal S\), we may therefore
	assume that
	$
	\Sigma_1=A_1
	$
	and that all remaining components are disjoint from this common
	sphere.
	
	\medskip
	\noindent
	\emph{Step 3: induction relative to the common sphere.}
	
	Cut \(Q\) along the common sphere
	$
	S=\Sigma_1=A_1
	$
	without capping the two resulting boundary components. Denote the
	resulting manifold by
	\[
	Q'
	=
	Q\setminus\operatorname{Int}\nu(S).
	\]
	Since the complement of each complete sphere system is connected,
	\(Q'\) is connected. Its boundary is
	\[
	\partial Q'
	=
	\partial Q\sqcup S^+\sqcup S^-,
	\]
	where \(S^+\) and \(S^-\) are the two boundary copies of \(S\).
	
	The remaining sphere systems
	\[
	\mathcal S'
	=
	\{\Sigma_2,\ldots,\Sigma_k\},
	\qquad
	\mathcal A'
	=
	\{A_2,\ldots,A_k\}
	\]
	lie in \(\operatorname{Int}Q'\). Their complements in \(Q'\) are
	connected, since they are precisely the original connected
	complements with the neighbourhood of the common sphere left as
	boundary.
	
	Let \(\widehat Q'\) be obtained by capping every boundary component
	of \(Q'\), including \(S^+\) and \(S^-\). Starting from
	\(\widehat C_{\mathcal S}\cong M'\), the manifold \(\widehat Q'\)
	is obtained by restoring only the \(k-1\) three-dimensional
	\(1\)-handles corresponding to
	\(\Sigma_2,\ldots,\Sigma_k\). Therefore,
	\[
	\widehat Q'
	\cong
	M'\#\bigl(\#_{k-1}(\s^1\times \s^2)\bigr).
	\]
	
	Every component of \(\mathcal S'\) and \(\mathcal A'\) is essential
	in \(\widehat Q'\). Indeed, its complement in \(Q'\), and hence in
	\(\widehat Q'\), is connected, so it is nonseparating.
	
	The induction hypothesis applied to \(Q'\) shows that
	\(\mathcal S'\) and \(\mathcal A'\) are related by:
	
	\begin{enumerate}
		\item ambient isotopies fixed on \(\partial Q'\);
		\item permutations;
		\item sphere slides among the remaining components;
		\item boundary slides over components of
		\[
		\partial Q'
		=
		\partial Q\sqcup S^+\sqcup S^-.
		\]
	\end{enumerate}
	
	Reglue the two boundary copies \(S^+\) and \(S^-\) to recover \(Q\).
	The isotopies fixed on \(\partial Q'\) lift to isotopies of \(Q\)
	fixing the common sphere \(S\). Sphere slides among the remaining
	components lift unchanged. Boundary slides over components of
	\(\partial Q\) remain boundary slides.
	
	Finally, a boundary slide over either \(S^+\) or \(S^-\) becomes,
	after regluing, a sphere slide over the common sphere
	$
	S=\Sigma_1=A_1.
	$
	Thus all moves supplied by the induction hypothesis lift to allowed
	moves in \(Q\), while the common sphere remains fixed.
	
	It follows that \(\mathcal S\) and \(\mathcal A\) are related by
	ambient isotopies fixed on \(\partial Q\), permutations, sphere
	slides, and boundary slides. This completes the induction.
\end{proof}


\begin{proof}[Proof of Theorem \ref{thm1}]
	
Since \(M'\) is irreducible, every smoothly embedded \(2\)-sphere in
\(M'\) bounds a \(3\)-ball. Hence
$
H^{\mathrm{sph}}_{2}(M')=0.
$

Let
$
M=M'\#\bigl(\#_{k}(\s^{1}\times \s^{2})\bigr),
$
and let \(s_i=\{\mathrm{pt}\}\times \s^{2}\) denote the standard
nonseparating sphere in the \(i\)-th \(\s^{1}\times \s^{2}\) summand. We
claim that
\[
H^{\mathrm{sph}}_{2}(M)
=
\bigoplus_{i=1}^{k}\mathbb Z[s_i].
\]

Indeed, let \(\Sigma\subset M\) be an oriented, smoothly embedded
\(2\)-sphere. Let
$
P=P_1\cup\cdots\cup P_k
$
be a collection of pairwise disjoint connected-sum spheres realizing the
decomposition
$
M=M'\#\bigl(\#_{k}(\s^{1}\times \s^{2})\bigr).
$
After an isotopy, we may assume that \(\Sigma\) is transverse to \(P\).

We now eliminate the circles of \(\Sigma\cap P\) by surgery. Choose a
circle of intersection that is innermost on some \(P_i\), and let
\(D\subset P_i\) be the disk that it bounds, with
\(\operatorname{Int}D\cap\Sigma=\varnothing\). Since \(\Sigma\cong S^2\),
this circle separates \(\Sigma\) into two disks. Replacing either of these
disks by a parallel copy of \(D\) produces two embedded \(2\)-spheres
\(\Sigma'\) and \(\Sigma''\) such that, after choosing orientations,
\[
[\Sigma]=[\Sigma']+[\Sigma'']
\qquad\text{in }H_2(M;\mathbb Z),
\]
and the total number of intersections with \(P\) decreases. Repeating
this procedure, we express \([\Sigma]\) as a sum of homology classes
represented by embedded \(2\)-spheres, each of which is disjoint from
\(P\). After capping the boundary spheres of the connected-sum pieces by
\(3\)-balls, each such sphere lies entirely in one of the summands
\(M'\) or \(\s^1\times \s^2\).

Every sphere lying in the \(M'\)-summand represents the zero class,
because \(M'\) is irreducible. On the other hand,
$
H_2(\s^1\times \s^2;\mathbb Z)\cong\mathbb Z
$
is generated by the class of \(\{\mathrm{pt}\}\times \s^2\). Therefore,
every spherical homology class in \(M\) is an integral linear
combination of \([s_1],\ldots,[s_k]\). Since the spheres
\(s_1,\ldots,s_k\) are themselves embedded, the reverse inclusion is
immediate. 

Moreover, the Mayer--Vietoris sequence for the connected-sum
decomposition gives
\[
H_2(M;\mathbb Z)
\cong
H_2(M';\mathbb Z)
\oplus
\bigoplus_{i=1}^{k}H_2(\s^1\times \s^2;\mathbb Z).
\]
Under this isomorphism, the class \([s_i]\) is the standard generator of
the \(i\)-th copy of
\[
H_2(\s^1\times \s^2;\mathbb Z)\cong\mathbb Z.
\]
Consequently, the classes
\([s_1],\ldots,[s_k]\) are linearly independent. Thus
\[
H^{\mathrm{sph}}_{2}(M)
=
\bigoplus_{i=1}^{k}\mathbb Z[s_i]
\cong\mathbb Z^k.
\]

	Since there are no $4$-handles,
	$
	\partial W_3=\partial W,
	$
	which is connected by hypothesis.
	
	
	\underline{(2)$\Rightarrow$(3)}
	
	Put \(M=\partial W_{2}\) and assume that
	$
	C
	=
	M\setminus
	\operatorname{Int}\nu(\Sigma_{1}\cup\cdots\cup\Sigma_{k})
	$
	is connected. Choose an orientation of each sphere \(\Sigma_i\). The boundary
	of \(C\) contains two copies \(\Sigma_i^{+}\) and \(\Sigma_i^{-}\) of each
	\(\Sigma_i\). Since \(C\) is connected, for every \(i\) we may choose a properly
	embedded arc
	$
	\alpha_i\subset C
	$
	whose endpoints lie on \(\Sigma_i^{-}\) and \(\Sigma_i^{+}\), respectively,
	and whose interior is contained in \(\operatorname{Int}C\). Closing \(\alpha_i\) by a normal interval in
	\(\nu(\Sigma_i)\cong \Sigma_i\times[-1,1]\), we obtain an oriented loop
	\(c_i\subset M\). The loops may be chosen transverse to the spheres and,
	after choosing their orientations appropriately, satisfy
	$
	[c_i]\cdot[\Sigma_j]=\delta_{ij}.
	$
	
	Consider the intersection homomorphism
	\[
	\Lambda\colon H^{\mathrm{sph}}_{2}(M)\longrightarrow \mathbb Z^{k},
	\qquad
	\Lambda(x)
	=
	\bigl([c_1]\cdot x,\ldots,[c_k]\cdot x\bigr).
	\]
	For every \(j\), we have
	$
	\Lambda([\Sigma_j])=e_j,
	$
	where \(e_j\) denotes the \(j\)-th standard basis vector of \(\mathbb Z^k\).
	Consequently, \(\Lambda\) is surjective.
	
	We have already established that
	$
	H^{\mathrm{sph}}_{2}(M)\cong \mathbb Z^{k}.
	$
	Thus \(\Lambda\) is a surjective homomorphism between free abelian groups of
	the same finite rank, and therefore it is an isomorphism. Since
	$
	[\Sigma_j]=\Lambda^{-1}(e_j)
	$
	for \(j=1,\ldots,k\), the classes
	$
	[\Sigma_1],\ldots,[\Sigma_k]
	$
	form a \(\mathbb Z\)-basis of \(H^{\mathrm{sph}}_{2}(M)\).


	\underline{(3)$\Rightarrow$(2)}

	Assume that
	$
	M\setminus
	\operatorname{Int}\nu(\Sigma_1\cup\cdots\cup\Sigma_k)
	$
	is disconnected, and let $U$ be the closure of one of its connected
	components. After choosing orientations, the oriented boundary of $U$
	gives a relation
	\[
	\sum_{i=1}^k\varepsilon_i[\Sigma_i]=0,
	\qquad
	\varepsilon_i\in\{-1,0,1\}.
	\]
	This relation is nontrivial. Indeed, consider the dual graph of the
	sphere system: its vertices are the components of the complement and
	its edges correspond to the spheres $\Sigma_i$. Since $M$ is connected,
	the dual graph is connected. Since the complement has more than one
	component, the vertex corresponding to $U$ is incident to an edge
	joining it to another vertex. The corresponding coefficient
	$\varepsilon_i$ is therefore equal to $\pm1$.
	
	Thus, the displayed relation is nontrivial, contradicting the linear
	independence of
	$
	[\Sigma_1],\ldots,[\Sigma_k].
	$
	Therefore, the complement is connected.

	\underline{(1)$\Rightarrow$(2)}
	
	Let
	$
	\mathcal A=\{A_1,\ldots,A_k\}
	$
	be the collection of the actual attaching spheres of the $k$
	$3$-handles of $W$ in
	$
	M=\partial W_2.
	$
	Since the handle decomposition of $W$ contains no $4$-handles, we have
	$
	\partial W_3=\partial W,
	$
	which is connected by hypothesis.
	
	Attaching the $3$-handles along
	$A_1,\ldots,A_k$ is equivalent, on the boundary, to performing
	simultaneous sphere surgery on these spheres. More precisely,
	$\partial W_3$ is obtained from
	$
	C_{\mathcal A}
	=
	M\setminus
	\operatorname{Int}\nu(A_1\cup\cdots\cup A_k)
	$
	by gluing a $3$-ball to each of its $2k$ spherical boundary
	components.
	
	If $C_{\mathcal A}$ were disconnected, gluing $3$-balls to its
	boundary components could not join two distinct connected components.
	It would follow that $\partial W_3$ is disconnected, contrary to the
	connectedness of $\partial W$. Therefore,
	$
	M\setminus
	\operatorname{Int}\nu(A_1\cup\cdots\cup A_k)
	$
	is connected.
	
	Now assume condition~(1). Since
	$
	M\setminus
	\operatorname{Int}\nu(A_1\cup\cdots\cup A_k)
	$
	is connected, the already proved implication $(2)\Rightarrow(3)$ shows
	that
	$
	[A_1],\ldots,[A_k]
	$
	form a $\mathbb Z$-basis of $H_2^{\mathrm{sph}}(M)$.
	
	Ambient isotopies and permutations preserve this property. Moreover,
	a $3$--$3$ handle slide of $A_i$ over $A_j$ replaces the homology
	class $[A_i]$ by
	$
	[A_i]\pm[A_j],
	$
	while leaving the other classes unchanged. This is an elementary
	unimodular change of basis.
	
	Therefore condition~(1) implies that
	$
	[\Sigma_1],\ldots,[\Sigma_k]
	$
	form a $\mathbb Z$-basis of $H_2^{\mathrm{sph}}(M)$. By the already
	proved implication $(3)\Rightarrow(2)$,
	$
	M\setminus
	\operatorname{Int}\nu(\Sigma_1\cup\cdots\cup\Sigma_k)
	$
	is connected.

	\underline{(3)$\Rightarrow$(1)}
	
	Put
	$
	M=\partial W_2,
	$
	and let
	$
	\mathcal S=\{\Sigma_1,\ldots,\Sigma_k\}.
	$
	Let
	$
	\mathcal A=\{A_1,\ldots,A_k\}
	$
	be the collection of the actual attaching spheres of the $k$
	$3$-handles of $W$.
	
	If $k=0$, both the implication and the final surgery statement are
	immediate. Hence, assume that $k>0$.
	
	By condition~(3), the classes
	$
	[\Sigma_1],\ldots,[\Sigma_k]
	$
	form a $\mathbb Z$-basis of
	$H_2^{\mathrm{sph}}(M)$. In particular, each class
	$[\Sigma_i]$ is nonzero, and hence each sphere $\Sigma_i$ is
	essential. Moreover, by the already proved implication
	$(3)\Rightarrow(2)$,
	$
	M\setminus
	\operatorname{Int}\nu(\Sigma_1\cup\cdots\cup\Sigma_k)
	$
	is connected.
	
	We next consider the actual attaching sphere system
	$\mathcal A$. Attaching the $3$-handles along
	$A_1,\ldots,A_k$ performs simultaneous sphere surgery on
	$\mathcal A$. Thus $\partial W_3=\partial W$ is obtained from
	$
	C_{\mathcal A}
	=
	M\setminus
	\operatorname{Int}\nu(A_1\cup\cdots\cup A_k)
	$
	by gluing a $3$-ball to each spherical boundary component.
	Since $\partial W$ is connected, $C_{\mathcal A}$ must be
	connected: gluing $3$-balls to boundary components cannot join
	distinct connected components.
	
	It follows also that every $A_i$ is essential. Indeed, if some
	$A_i$ bounded a $3$-ball in $M$, then $A_i$ would separate $M$,
	and, since the remaining attaching spheres are disjoint from
	$A_i$, the complement $C_{\mathcal A}$ would be disconnected.
	
Consequently, \(\mathcal S\) and \(\mathcal A\) are two collections
of \(k\) pairwise disjoint essential \(2\)-spheres in
$
M\cong M'\#\bigl(\#_k(\s^1\times \s^2)\bigr),
$
and both collections have connected complements. By
Lemma~\ref{lem:relative-complete-sphere-systems}, they are related by ambient
isotopies, a permutation of their components, and sphere slides.

A sphere slide of one attaching sphere over another is precisely the
operation induced on the attaching spheres by a \(3\)--\(3\) handle
slide. Therefore, up to a permutation and \(3\)--\(3\) handle slides,
the spheres \(\Sigma_1,\ldots,\Sigma_k\) may be isotoped to the actual
attaching spheres \(A_1,\ldots,A_k\).

	This proves condition~(1).
	
	Finally, cutting $M$ along the $\Sigma_i$ and capping the resulting $2k$
	spherical boundary components by $3$-balls gives a connected
	$3$-manifold $M_\Sigma$. Reversing this operation attaches $k$
	$3$-dimensional $1$-handles, hence
	\[
	M \cong M_{\Sigma} \# \Bigl(\#_k(\s^1\times \s^2)\Bigr).
	\]
	Comparing with the given prime decomposition of $M$ and applying the
	Kneser--Milnor uniqueness theorem yields $M_\Sigma\cong M'$.
\end{proof}

Theorem \ref{thm1} characterizes the geometric sphere systems that represent the attaching system: a collection of $k$ pairwise disjoint embedded spheres is equivalent, up to isotopy, permutation, and sphere slides, to the actual attaching system if and only if its homology classes form a basis of $H_2^{\mathrm{sph}}(\partial W_2)$.
Under the assumptions of Theorem \ref{thm1}, we thus obtain a uniqueness result for $3$-handle attachments.

\begin{corollary}\label{cor03} Let $N$ be a compact smooth $4$-manifold with the boundary $\partial N=M'\#(\#_k\;\s^{1}\times \s^{2})$ with $M'$ irreducible. Denote by $N'=N\cup h_3^{(1)}\cup\ldots \cup h_3^{(k)}$ a manifold, obtained by attaching $k$ four-dimensional $3$-handles to $\partial N$. If $\partial N'$ is connected and the homology classes of the attaching spheres $\Sigma_1,\ldots , \Sigma_k$ form a basis of $H_2^{\mathrm{sph}}(\partial N)$, then the resulting manifold $N'$ is unique up to diffeomorphism.
\end{corollary}

\begin{remark}\label{rem01} Let $W$ be a compact smooth $4$-manifold with the boundary $\partial W_2=M'\#(\#_k\;\s^{1}\times \s^{2})$ with $M'$ irreducible. Assume that the number of $3$-handles in a handle decomposition (with no $4$-handles) of $W$ equals $k$. To visualize the attaching $2$-spheres, we can trace the basis [$\{point_1\}\times \s^2$],\ldots, $[\{point_k\}\times \s^2]$ back to the Kirby diagram for $W$, where $\{point_i\}\times \s^2$ lies in the $i$-th summand of the part $(\#_{k}\s ^{1}\times \s ^{2})$ of the boundary.
	
In the extended diagram, that contains drawings of pairwise disjoint essential spheres $\Sigma_j$, verifying that these spheres are isotopic to the actual $3$-handle attaching spheres (up to permutation and handle slides) reduces to finding $k$ loops $c_1,\dots,c_k\subset \partial W_2$, such that $[c_i]\cdot [\Sigma_j] = \delta_{ij}$ for all $i,j$ (where $\delta_{ij}$ is the Kronecker delta). This also follows from duality and intersection pairing; see \cite[II. \S 8]{NR04}.

\end{remark}


\subsection{An example}\label{An_example}

Let $E_2$ denote the oriented $\di^2$-bundle over $\s^2$ with Euler number $2$. Its Kirby diagram consists of a single $2$-framed unknot, and $\partial E_2\cong L(2,1)\cong \mathbb{R}P^3.$

Define the $0$--$1$--$2$ handlebody $W_2 = E_2\natural(\s^1\times \di^3),$
where $\natural$ denotes boundary connected sum. Equivalently, $W_2$ is obtained from a single $0$-handle by attaching one $1$-handle and one $2$-handle, where the attaching circle of the $2$-handle is a $2$-framed unknot split from the dotted circle representing the $1$-handle.

\begin{figure}[h!t]
	\begin{center}
		\begin{overpic}[scale=0.2]{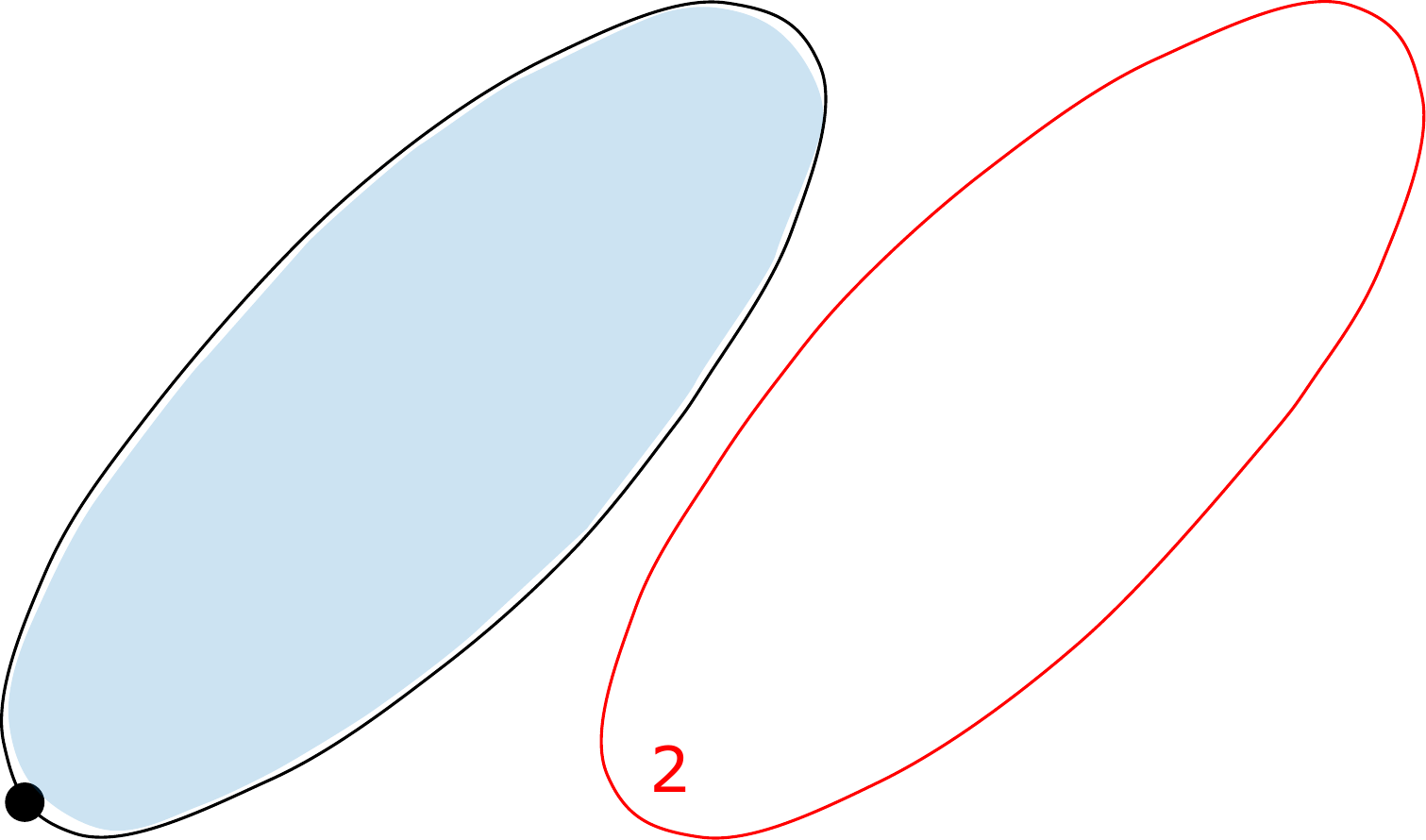}
		\end{overpic}
		\caption{A diagram $\widehat{\mathcal{D}}(W)$. \label{exp52c}}
	\end{center}
\end{figure}

\begin{figure}[ht]
	\centering
	\begin{overpic}[scale=0.85]{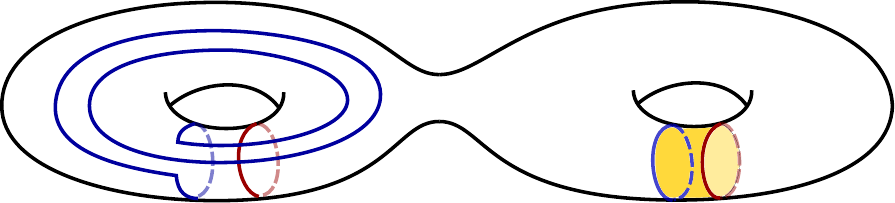}
		\put(24,  2){$\alpha _1$}
		\put(43, 10){$\beta _{1}$}
		\put(79,  4){$\alpha _2$}
		\put(69, 4){$\beta _{2}$}
	\end{overpic}
	\caption{The attaching sphere of the $3$-handle in the Heegaard diagram for
		$\partial W_2$. \label{heeg_diagram3}}
\end{figure}

The boundary connected-sum description gives
$$
\partial W_2
\cong
\partial E_2\#\partial(\s^1\times \di^3)
\cong
\mathbb{R}P^3\#(\s^1\times \s^2).
$$

Let
$\Sigma\subset\partial W_2$
be the standard nonseparating sphere in the $\s^1\times \s^2$ summand;
that is,

$\Sigma=\{\mathrm{pt}\}\times \s^2
\subset \s^1\times \s^2.$
Equivalently, $\Sigma$ is the belt sphere of the $1$-handle. In the
extended Kirby diagram, the visible part of $\Sigma$ is represented by
a shaded disc bounded by the dotted component.

We now define
$
W=W_2\cup_{\Sigma}h_3,
$
where $h_3$ is a $4$-dimensional $3$-handle attached along $\Sigma$.
Surgery on the nonseparating sphere $\Sigma$ removes the
$\s^1\times \s^2$ summand from the boundary. Consequently,
$
\partial W\cong\mathbb{R}P^3.
$
In particular, $\partial W$ is connected.

The handle decomposition of $W$ has one handle of each index
$0,1,2,3$ and no $4$-handles:
$
|h_0|=|h_1|=|h_2|=|h_3|=1,
\;
|h_4|=0.
$
Moreover,
$$
H_2^{\mathrm{sph}}(\partial W_2)
\cong
H_2(\s^1\times \s^2;\mathbb Z)
\cong
\mathbb Z,
$$
and $[\Sigma]$ is a generator. Thus
\[
\operatorname{rank}
H_2^{\mathrm{sph}}(\partial W_2)=1,
\]
which equals the number of $3$-handles.

A small meridian of the dotted component intersects $\Sigma$
algebraically once. Hence the homology class $[\Sigma]$ is dual to the
class represented by this meridian and forms a basis of
$H_2^{\mathrm{sph}}(\partial W_2)$. 
The intersection computation verifies diagrammatically that $[\Sigma]$ is the required generator of $H_2^{\mathrm{sph}}(\partial W_2)$, in accordance with Remark \ref{rem01}.

Finally,
$
\pi_1(W)\cong\pi_1(W_2)\cong\mathbb Z,
$
because the split $2$-handle does not introduce a relation involving
the generator corresponding to the $1$-handle, and the attachment of a
$3$-handle does not change the fundamental group. Therefore, this
example is not covered by \cite{Tr82}. It is also not covered directly
by \cite{LP72}, since
$
\partial W_2
\cong
\mathbb{R}P^3\#(\s^1\times \s^2)
$
has a nontrivial irreducible summand.

In Figure \ref{exp52c}, we show the extended Kirby diagram $\widehat{\mathcal{D}}(W)$ with the visible part of the essential $2$-sphere drawn in blue.
In Figure \ref{heeg_diagram3}, the attaching sphere (yellow) of the $3$-handle in a Heegaard diagram for $\partial W_{2}$ is shown.

\section*{Acknowledgements}

Both authors contributed equally to this work. The first author was supported by the Slovenian Research Agency grants P1-0292, N1-0278 and J1-4031.



\end{document}